\documentclass[11pt,twoside,leqno]{article}
\usepackage{amsfonts,amssymb,amscd,amsmath,amsthm,enumerate,verbatim,calc,latexsym}
\usepackage[colorlinks=true,linkcolor=blue,urlcolor=blue,citecolor=blue,anchorcolor=blue,bookmarksnumbered=blue]{hyperref}
\usepackage[all]{xy}
\usepackage[dvipsnames]{xcolor}
\newtheorem{theorem}{Theorem}[section]
\newtheorem{proposition}[theorem]{Proposition}
\newtheorem{corollary}[theorem]{Corollary}
\newtheorem{lemma}[theorem]{Lemma}
\newtheorem{definition}[theorem]{Definition}

\newtheorem{example}[theorem]{Example}
\newtheorem{remark}[theorem]{Remark}

\numberwithin{equation}{section}
\begin{document}
\title{\bf Foxby equivalence relative to $C$-$fp_{n}$-injective and $C$-$fp_{n}$-flat  modules}
\author{
\\{\bf Mostafa Amini}\\
\small Department of Mathematics, Payame Noor University, Tehran, Iran\\
\small E-mail: dr.mostafa56@pnu.ac.ir\\
\\{\bf Alireza Vahidi}\\
\small Department of Mathematics, Payame Noor University, Tehran, Iran\\
\small E-mail: vahidi.ar@pnu.ac.ir\\
\\{\bf Farideh Rezaei}\\
\small Department of Mathematics, Payame Noor University, Tehran, Iran\\
\small E-mail: fa.rezaei@student.pnu.ac.ir
}

\date{}
\maketitle
\begin{abstract}
Let $R$ and $S$ be rings, $C= {}_SC_R$ a (faithfully) semidualizing bimodule, and $n$ a positive integer or $n=\infty$. In this paper,  we  introduce the concepts of $C$-$fp_n$-injective $R$-modules and $C$-$fp_n$-flat $S$-modules as a common generalization of some known modules such as $C$-$FP_{n}$-injective (resp. $C$-weak injective) $R$-modules and $C$-$FP_{n}$-flat (resp. $C$-weak flat) $S$-modules. Then we investigate $C$-$fp_{n}$-injective and $C$-$fp_{n}$-flat dimensions of modules, where the classes of these modules, namely $Cfp_nI(R)_{\leq k}$ and $Cfp_nF(S)_{\leq k}$, respectively.  We study Foxby equivalence relative to these classes, and also the existence of $Cfp_nI(R)_{\leq k}$ and $Cfp_nF(S)_{\leq k}$ preenvelopes and covers. Finally, we study the exchange properties of these classes, as well as preenvelopes (resp. precovers) and  Foxby equivalence, under almost excellent extensions of rings. \\
\vspace*{-0.25cm}\\
{\bf Keywords:} $C$-$fp_n$-flat modules; $C$-$fp_n$-injective modules; Foxby equivalence; semidualizing bimodules.\\
{\bf 2020 Mathematics Subject Classification:} 16E10; 16E30; 16E65; 16P70.
\end{abstract}
\section{Introduction}\label{1}
Throughout this paper, $n$ is a positive integer, $R$ and $S$ are fixed associative rings with unites, and all $R$- or $S$-modules are understood to be unital left $R$- or $S$-modules (unless specified otherwise). ${}_SM$ (resp. $M_R$) is used to denote that $M$ is a left $S$-module (resp. right $R$-module). Also, ${}_SM_R$ is used to denote that $M$ is an $(S, R)$-bimodule which means that $M$ is both a left $S$-module and a right $R$-module, and these structures are compatible. Right $R$- or $S$-modules are identified with left modules over the opposite rings $R^{op}$ and $S^{op}$.

Wei and Zhang in \cite{JZW} introduced the notion of $fp_n$-injective (resp. $fp_n$-flat) modules as a generalization of $fp$-injective and $FP_n$-injective (resp. $fp$-flat and $FP_n$-flat) modules, where $fp$-injective and $fp$-flat modules introduced by Garkusha and Generalov in \cite{E.SS}, and also $FP_n$-injective and $FP_n$-flat modules introduced by Bravo and P\'{e}rez in \cite{ARB}.

Over a commutative Noetherian ring $R$, a semidualizing module for $R$ is a finitely generated $R$-module $C$ with $\operatorname{Hom}_{R}(C, C)$  canonically isomorphic to $R$ and $\operatorname{Ext}_{R}^{i}(C, C)= 0$ for all $i\geq 1$. Semidualizing modules (under different names) were independently studied by Foxby in \cite{H.B}, Vasconcelos in \cite{WWV}, and Golod in \cite{E.SG}. Araya et al., in \cite{AR}, extended the notion of semidualizing modules to a pair of non-commutative, but Noetherian rings. Holm and White, in \cite{HW}, generalized the notion of a semidualizing module to general associative rings, and defined and studied Auslander and Bass classes under a semidualizing bimodule $C$, and then introduced the notions of $C$-flat, $C$-projective, and $C$-injective modules, where $C= {}_SC_R$ stands for a semidualizing bimodule.

In \cite{WGZ}, Wu and Gao introduced the notion of $C$-$FP_n$-injective (resp. $C$-$FP_n$-flat) modules as a common generalization of some known modules such as $C$-injective, $C$-$FP$-injective and $C$-weak injective (resp. $C$-flat and  $C$-weak flat) modules (see \cite{Z.TT,HW,ZO}). Furthermore, they investigated Foxby equivalence relative to $C$-$FP_n$-injective $R$-modules and $C$-$FP_n$-flat $S$-modules, proved that the classes $\mathcal{FI}^{n}_C(R)$ and $\mathcal{FF}^{n}_C(S)$ are  preenveloping and covering, and found that when these classes are closed under extensions, cokernels of monomorphisms, and kernels of epimorphisms, where $\mathcal{FI}^{n}_C(R)$ and $\mathcal{FF}^{n}_C(S)$ are the classes of $C$-$FP_n$-injective $R$-modules and $C$-$FP_n$-flat $S$-modules, respectively.

Recently, the homological theory for injective modules and flat modules with respect to semidualizing bimodules has became an important area of research (see for example \cite{B1, B2, Z.MT, Z.TT, HW, WGZ}). In this paper, we introduce and study the notion of $C$-$fp_n$-injective (resp. $C$-$fp_n$-flat) modules as a  common generalization of $C$-weak injective and $C$-$FP_n$-injective (resp. $C$-weak flat and $C$-$FP_n$-flat) modules. 


In Section \ref{2}, we state some fundamental notions and some preliminary results. Then we present some features of  the Auslander and Bass classes, and modules of $fp_n$-injective and $fp_n$-flat dimension at most $k$. In Section \ref{3}, first we introduce $C$-$fp_n$-injective $R$-modules and $C$-$fp_n$-flat $S$-modules, and then we give some homological relationships between the classes $fp_nI(S)_{\leq k}$, $fp_nF(R)_{\leq k}$, $Cfp_nI(R)_{\leq k}$, $Cfp_nF(S)_{\leq k}$, $\mathcal{A}_{C}(R)$, and $\mathcal{B}_{C}(S)$, where these classes are the class of $S$-modules with $fp_n$-injective dimention at most $k$, the class of $R$-modules with $fp_n$-flat dimention at most $k$, the class of $R$-modules with $C$-$fp_n$-injective dimention at most $k$,  the class of $S$-modules with $C$-$fp_n$-injective dimention at most $k$, the Auslander class, and  the Bass class under faithfully semidualizing bimodules $C$, respectively. Among other results, we prove that (i) Foxby equivalence relative to these classes, (ii) for an $R$-module $M$ (resp. $S$-module $N$), $M\in Cfp_nI(R)_{\leq k}$ (resp. $N\in Cfp_nF(S)_{\leq k}$) if and only if $M\in \mathcal{A}_{C}(R)$ (resp. $N\in \mathcal{B}_{C}(S)$) and $C\otimes_{R} M\in fp_nI(S)_{\leq k}$ (resp. $\operatorname{Hom}_{S}(C, N)\in fp_nF(R)_{\leq k}$), and (iii) the classes $Cfp_{n}I(R)$ and $Cfp_{n}F(S)$ are preenveloping and covering. Section \ref{4} considering  faithfully semidualizing modules $C$ is devoted to the exchange properties of these classes under chang of rings. For example, let $S\geq R$ be an almost excellent extension. Then we  show that (i) the classe $(C\otimes_{R}S)fp_{n}I(R)$ and $(C\otimes_{R}S)fp_{n}F(R)$ are preenveloping and precovering, (ii)  if $M\in \mathcal{A}_{C}(R)$, then $(S\otimes_{R}M)\in \mathcal{A}_{C\otimes_{R}S}(S)$;
 if $M\in \mathcal{B}_{C}(R)$, then ${\rm Hom}_{R}(S,M)\in \mathcal{B}_{C\otimes_{R}S}(S)$, and (ii) existence Foxby equivalence relative to the classes $(C\otimes_{R}S)fp_{n}I(R)$, $(C\otimes_{R}S)fp_{n}F(R)$, $\mathcal{A}_{C\otimes_{R}S}(S)$ and $\mathcal{B}_{C\otimes_{R}S}(S)$.

\section{Preliminaries}\label{2}



In this section, some fundamental concepts and notations are stated.

\begin{definition}\label{2-1}
\emph{(see \cite[Section 1 and Definitions 2.2, 3.1, and 3.2]{ARB}, \cite[Definitions 2.8 and 2.18]{.TZ}, \cite[Definition 2.1]{Z.G2},  \cite[Definition 2.1]{Z.W}, \cite[1.2, 1.3, and 1.4]{Z.TT},  \cite[Definition 3.1]{WGZ}, \cite[Definition 2.1]{Z.TT} and \cite[Definition 2.1]{JZW})}
\begin{itemize}
\item[\emph{(i)}] \emph{An $R$-module $M$ is called \textit{finitely $n$-presented} if there exists an exact sequence
\[F_n\longrightarrow F_{n- 1}\longrightarrow \cdots\longrightarrow F_i\longrightarrow \cdots\longrightarrow F_1\longrightarrow F_0\longrightarrow M\longrightarrow 0,\]
where each $F_i$ is a finitely generated free $($equivalently, finitely generated projective$)$ $R$-module. A ring $R$ is called \textit{$n$-coherent} if every finitely $n$-presented $R$-module is finitely $(n+ 1)$-presented;}

\item[\emph{(ii)}] \emph{An $R$-module $M$ is called  \textit{$FP_n$-injective} or \textit{$(n ,0)$-injective} $($resp. \textit{$FP_n$-flat} or \textit{$(n, 0)$-flat}$)$ if $\operatorname{Ext}_{R}^{1}(L, M)= 0$ $($resp. $\operatorname{Tor}_{1}^{R}(L, M)= 0)$ for any finitely $n$-presented $R$-module $($resp. $R^{op}$-module$)$ $L$. $\mathcal{FP}_n$-$\operatorname{Inj}(R)$ and $\mathcal{FP}_n$-$\operatorname{Flat}(R)$ denote the class of $FP_n$-injective $R$-modules and the class of $FP_n$-Flat $R$-modules, respectively;}

\item[\emph{(iii)}] \emph{The \textit{$FP_n$-injective dimension} $($or \textit{$(n, 0)$-injective dimension}$)$ and the \textit{$FP_n$-flat dimension} $($or \textit{$(n, 0)$-flat dimension}$)$ of an $R$-module $M$ are defined by
\[\mathcal{FP}_n.\operatorname{id}_R(M)= \inf\{k: \operatorname{Ext}_R^{k+ 1}(L, M)= 0\ \text{for every finitely $n$-presented $R$-module}\ L\}\]
and
\[\mathcal{FP}_n.\operatorname{fd}_R(M)= \inf\{k: \operatorname{Tor}^R_{k+ 1}(L, M)=0\ \text{for every finitely $n$-presented $R^{op}$-module}\ L\},\]
respectively;}

\item[\emph{(iv)}] \emph{A \textit{degreewise finite projective resolution (or, super finitely presented)} of an $R$-module $M$ is a projective resolution
of $M$:
\[\cdots\longrightarrow P_{i+ 1}\longrightarrow P_i\longrightarrow P_{i- 1}\longrightarrow\cdots\longrightarrow P_1\longrightarrow P_0\longrightarrow U\longrightarrow 0,\]
where each $P_i$ is a finitely generated projective $($equivalently, finitely generated free$)$ $R$-module;}

\item[\emph{(v)}] \emph{An $R$-module $M$ is called  \textit{weak injective} $($resp. \textit{weak flat}$)$  if $\operatorname{Ext}_{R}^{1}(U, M)= 0$ $($resp. $\operatorname{Tor}_{1}^{R}(U, M)= 0)$ for any super finitely presented $R$-module $($resp. $R^{op}$-module$)$ $U$. }

\item[\emph{(vi)}] \emph{An $(S, R)$-bimodule $C= {}_SC_R$ is \textit{semidualizing} if the following conditions hold:
\begin{itemize}
\item[$(a_1)$]  ${}_SC$ admits a degreewise finite $S$-projective resolution;
\item[$(a_2)$]  $C_R$ admits a degreewise finite $R^{op}$-projective resolution;
\item[$(b_1)$]  The homothety map ${}_S\gamma: {{}_SS_S}\longrightarrow \operatorname{Hom}_{R^{op}}(C, C)$ is an isomorphism;
\item[$(b_2)$]  The homothety map $\gamma_{R}: {}_RR_R\longrightarrow \operatorname{Hom}_{S}(C, C)$ is an isomorphism;
\item[$(c_1)$] $\operatorname{Ext}_{S}^{i}(C, C)= 0$ for all $i\geq 1$;
\item[$(c_2)$] $\operatorname{Ext}_{R^{op}}^{i}(C, C)= 0$ for all $i\geq 1$.
\end{itemize}
A semidualizing bimodule ${}_SC_R$ is \textit{faithfully semidualizing} if it satisfies the following conditions for all modules ${}_SN$ and $M_R$:
\begin{itemize}
\item[$(1)$] If $\operatorname{Hom}_{S}(C, N)= 0$, then $N= 0$;
\item[$(2)$] If $\operatorname{Hom}_{R^{op}}(C, M)= 0$, then $M= 0$.
\end{itemize}
By \cite[Proposition 3.2]{HW}, there exist many examples of faithfully semidualizing bimodules were provided over a wide class of non-commutative rings;}

\item[\emph{(vii)}] \emph{The \textit{Auslander class} $\mathcal{A}_{C}(R)$ with respect to $C$ consists of all $R$-modules $M$ satisfying the following conditions:
\begin{itemize}
\item[$(A_1)$]  $\operatorname{Tor}_{i}^{R}(C, M)= 0$ for all $i\geq 1$;
\item[$(A_2)$]  $\operatorname{Ext}_{S}^{i}(C, C\otimes_{R} M)= 0$ for all $i\geq 1$;
\item[$(A_3)$]  The natural evaluation homomorphism $\mu_{M}: M\longrightarrow \operatorname{Hom}_{S}(C, C\otimes_{R} M)$ is an isomorphism $($of $R$-modules$)$.
\end{itemize}
The \textit{Bass class} $\mathcal{B}_{C}(S)$ with respect to $C$ consists of all $S$-modules $N$ satisfying the following conditions:
\begin{itemize}
\item[$(B_1)$]  $\operatorname{Ext}_{S}^{i}(C ,N)= 0$ for all $i\geq 1$;
\item[$(B_2) $]  $\operatorname{Tor}_{i}^{R}(C, \operatorname{Hom}_{S}(C, N))= 0$ for all $i\geq 1$;
\item[$(B_3)$]  The natural evaluation homomorphism $\nu_{N}: C\otimes_{R} \operatorname{Hom}_{S}(C, N)\longrightarrow N$ is an isomorphism $($of $S$-modules$)$.
\end{itemize}
It is an important property of Auslander and Bass classes that they are equivalent under the pair of functors:
\[\xymatrix@C=3cm{\mathcal{A}_{C}(R)\ar@<1ex>[r]^{C\otimes_{R} -}_{\sim}&\ar@<1ex>[l]^{\operatorname{Hom}_{S}(C, -)}\mathcal{B}_{C}(S)}\]
$($see \cite[Proposition 4.1]{HW}$)$;}

\item[\emph{(viii)}] \emph{An $R$-module is called \textit{$C$-weak injective} if it has the form $\operatorname{Hom}_{S}(C, X)$ for some weak injective  $S$-module $X$. An $S$-module is called \textit{$C$-weak flat} if it has the form $C\otimes_{R} Y$ for some  weak flat $R$-module $Y$;}

\item[\emph{(ix)}] \emph{An $R$-module is called \textit{$C$-$FP_n$-injective} if it has the form $\operatorname{Hom}_{S}(C, X)$ for some $FP_n$-injective  $S$-module $X$. An $S$-module is called \textit{$C$-$FP_n$-flat} if it has the form $C\otimes_{R} Y$ for some $FP_n$-flat $R$-module $Y$;}

\item[\emph{(x)}]  \emph{An $R$-module $M$ is called \textit{$fp_n$-injective} $($resp. \textit{$fp_n$-flat}$)$ if for every exact sequence $0\longrightarrow K\longrightarrow L$ with $K$ and $L$ are finitely $n$-presented $R$-modules $($resp. $R^{op}$-modules$)$, the induced sequence $\operatorname{Hom}_{R}(L, M)\longrightarrow \operatorname{Hom}_{R}(K, M)\longrightarrow 0$ $($resp. $0\longrightarrow K\otimes_{R} M\longrightarrow L\otimes_{R} M)$ is exact. $fp_nI(R)$ and $fp_nF(R)$ denote the class of $fp_n$-injective $R$-modules and the class of $fp_n$-Flat $R$-modules, respectively.}
\end{itemize}
\end{definition}

By \cite[Proposition 1.7(1)]{ARB}, every $FP_n$-injective (resp. $FP_n$-flat) module is $fp_m$-injective (resp. $fp_m$-flat) for any $m\geq n$.  But not conversely, see Example \ref{2-1-2}. The  Proposition \ref{2-2} shows that the converse is also true over $n$-coherent rings.

\begin{definition}\label{2-1}
{\rm Let $\mathcal{Y}=\cdots\stackrel{f_2}\longrightarrow F_1\stackrel{f_1}\longrightarrow F_0\stackrel{f_0}\longrightarrow U\longrightarrow 0,$ be an exact sequence of projective $R$-modules $F_i$. Then $\mathcal{Y}$ is called $\mathcal{Y}$-finitely presented $($equivalently, super finitely presented in  \cite{Z.G2}$)$  if $U$  and ${\rm Ker} f_{i}$ are finitely presented for any $i\geq 0$.} 
\end{definition}

\begin{proposition}\label{3-5-A-B}
Let $C$ be a  semidualizing module. Then the following assertions hold true:
\begin{itemize}
\item[\emph{(i)}] $M\in\mathcal{A}_{C}(R)$ if and only if $M^{*}\in\mathcal{B}_{C}(R^{op})$;
\item[\emph{(ii)}] $M\in\mathcal{B}_{C}(R)$ if and only if $M^{*}\in\mathcal{A}_{C}(R^{op})$.
\end{itemize}
\end{proposition}

\begin{proof}
(i). $(\Rightarrow)$ Assume that $M\in\mathcal{A}_{C}(R)$. Then by \cite[Lemma 3.53]{Rot2}, there is a $(-\otimes_{R}M)^{*}$-exact exact sequence $\mathcal{Y}=\cdots\longrightarrow F_1\longrightarrow F_0\longrightarrow C\longrightarrow 0,$ with each $F_{i}$ is finitely generated and free. So  by \cite[Theorem 2.76 ]{Rot2}, it is easy
to check that $0=\operatorname{Tor}_{i}^{R}(C, M)^{*}\cong \operatorname{Ext}_{R^{op}}^{i}(C, M^{*})$ for any $i\geq 1.$

On the other hand, we have $\operatorname{Ext}_{R}^{i}(C, C\otimes_{R}M)=0,$ and  so ${\rm Hom}_{R}(\mathcal{Y},C\otimes_{R}M)$ is exact. Since $\mathcal{Y}$ is $\mathcal{Y}$-finitely presented, then by \cite[Lemmas 3.53 and 3.55 and Theorem 2.76]{Rot2},  we deduce that  ${\rm Hom}_{R}(\mathcal{Y},C\otimes_{R}M)$-exact if and only if ${\rm Hom}_{R^{op}}(\mathcal{Y},C\otimes_{R}M)^{*}$-exact if and only if $\mathcal{Y}\otimes_{R^{op}}(C\otimes_{R}M)^{*}$-exact if and only if $\mathcal{Y}\otimes_{R^{op}}{\rm Hom}_{R^{op}}( C,M^{*})$-exact. Hence $\operatorname{Tor}_{i}^{R^{op}}(C, \operatorname{Hom}_{R^{op}}(C, M^{*}))= 0$ for all $i\geq 1$.    Also, we have $M\cong{\rm Hom}_{R}(C,C\otimes_{R}M)$. So by \cite[Lemma 3.55]{Rot2}, $M^{*}\cong{\rm Hom}_{R^{op}}(C,C\otimes_{R}M)^{*}\cong C\otimes_{R^{op}}(C\otimes_{R}M)^{*}\cong C\otimes_{R^{op}}{\rm Hom}_{R^{op}}(C,M^{*})$.  Then, it follows that $M^{*}\in\mathcal{B}_{C}(R^{op})$.

$(\Leftarrow)$ Let $M^{*}\in\mathcal{B}_{C}(R^{op})$.  Then  there is a ${\rm Hom}_{R^{op}}(-,M^*)$-exact exact sequence $\mathcal{Y}=\cdots\longrightarrow F_1\longrightarrow F_0\longrightarrow C\longrightarrow 0,$ with each $F_{i}$ is finitely generated and free. So by \cite[Theorem 2.76 and Lemma 3.53 ]{Rot2}, ${\rm Hom}_{R^{op}}(\mathcal{Y},M^{*})$-exact if and only if $(\mathcal{Y}\otimes_{R^{op}}M)^{*}$-exact if and only if $(\mathcal{Y}\otimes_{R}M)$-exact. So  $\operatorname{Tor}_{i}^{R}(C, M)=0$ for any $i\geq 1$. Also, we have $\operatorname{Tor}_{i}^{R^{op}}(\mathcal{Y}, {\rm Hom}_{R^{op}}(C,M^{*}))=0$ for any $i\geq 1$. Then  since $\mathcal{Y}$ is $\mathcal{Y}$-finitely presented, $\mathcal{Y}\otimes_{R^{op}}{\rm Hom}_{R^{op}}(C,M^{*})$-exact if and only if $\mathcal{Y}\otimes_{R^{op}}(C\otimes_{R}M)^{*}$-exact if and only ${\rm Hom}_{R^{op}}(\mathcal{Y},C\otimes_{R}M)^{*}$-exact if and only if  ${\rm Hom}_{R}(\mathcal{Y},C\otimes_{R}M)$-exact, and so $\operatorname{Ext}_{R}^{i}(C, C\otimes_{R}M)=0$ for any $i\geq 1$.  We have $M^{*}\cong C\otimes_{R^{op}}{\rm Hom}_{R^{op}}(C,M^{*})\cong C\otimes_{R^{op}}(C\otimes_{R}M)^{*}\cong {\rm Hom}_{R}(C,C\otimes_{R}M)^{*}$, and so $M\cong{\rm Hom}_{R}(C,C\otimes_{R}M)$. Consequently,  $M\in\mathcal{A}_{C}(R)$.

(ii). This is similar to that of (i).
\end{proof}

\begin{proposition}\label{2-2}
Let $R$ $($resp. $R^{op})$ be an $n$-coherent ring and $M$ an $R$-module. Then $M$ is $fp_m$-injective $($resp. $fp_m$-flat$)$ if and only if $M$ is $FP_n$-injective $($resp. $FP_n$-flat$)$  for any $m\geq n$.
\end{proposition}

\begin{proof}
Assume that $M$ is an $fp_m$-injective $($resp. $fp_m$-flat$)$ $R$-module and $L$ is a finitely $n$-presented $R$-module $($resp. $R^{op}$-module$)$. Since $R$ $($resp. $R^{op})$ is an $n$-coherent ring, there is an exact sequence
\[0\longrightarrow K_0\longrightarrow F_0\longrightarrow L\longrightarrow 0\]
of $R$-modules $($resp. $R^{op}$-modules$)$, where $K_0$ and $F_0$ are finitely $m$-presented. Thus we get $\operatorname{Ext}_{R}^{1}(L, M)= 0$ (resp. $\operatorname{Tor}_{1}^{R}(L, M)= 0$) by applying the derived functors of $\operatorname{Hom}_{R}(-, M)$ (resp. $-\otimes_{R} M$) to the above short exact sequence. Hence $M$ is an $FP_n$-injective $($resp.  $FP_n$-flat$)$ $R$-module.
\end{proof}
The following lemmas will be useful in the proof of the first main result of this section.

\begin{lemma}\label{3-3}
Suppose that $M$ is an $fp_n$-injective $($resp. $fp_n$-flat$)$ $R$-module and that
\[0\longrightarrow K\longrightarrow F\longrightarrow L\longrightarrow 0\]
is a short exact sequence of $R$-modules $($resp. $R^{op}$-modules$)$ such that $K$ is finitely $n$-presented and $F$ is finitely generated and free. Then $\operatorname{Ext}^{1}_R(L, M)= 0$ $($resp. $\operatorname{Tor}_{1}^R(L, M)= 0)$.
\end{lemma}

\begin{proof}
By applying the derived functors of $\operatorname{Hom}_{R}(-, M)$ (resp. $-\otimes_{R} M$) to the above short exact sequence, the assertion follows.
\end{proof}


\begin{lemma}\label{3-4}
Suppose that $M$ is an $fp_n$-injective $($resp. $fp_n$-flat$)$ $R$-module and that
\[0\longrightarrow X_0\longrightarrow X_1\longrightarrow X_2\longrightarrow \cdots\longrightarrow X_{j- 1}\longrightarrow X_j\longrightarrow X_{j+ 1}\longrightarrow \cdots\]
is an exact sequence of $R$-modules $($resp. $R^{op}$-modules$)$ such that $X_j$ is finitely $n$-presented for all $j\geq 0$. Then  the sequence
\[\cdots\longrightarrow \operatorname{Hom}_{R}(X_j, M)\longrightarrow \cdots\longrightarrow \operatorname{Hom}_{R}(X_2, M)\longrightarrow \operatorname{Hom}_{R}(X_1, M)\longrightarrow \operatorname{Hom}_{R}(X_0, M)\longrightarrow 0\]
$($resp.
\[0\longrightarrow X_0\otimes_{R} M\longrightarrow X_1\otimes_{R} M\longrightarrow X_2\otimes_{R} M\longrightarrow \cdots\longrightarrow X_j\otimes_{R} M\longrightarrow \cdots)\]
is exact.
\end{lemma}

\begin{proof}
Assume that $C_j= \operatorname{Coker}(X_{j- 1}\longrightarrow X_{j})$ for all $j\geq 1$. Then there exist short exact sequences
\[0\longrightarrow X_0\longrightarrow X_1\longrightarrow C_1\longrightarrow 0\]
and
\[0\longrightarrow C_{j- 1}\longrightarrow X_{j}\longrightarrow C_{j}\longrightarrow 0,\]
for all $j\geq 2$. Thus $C_j$ is finitely $n$-presented for all $j\geq 1$ from \cite[Proposition 1.7(1)]{ARB} and using an induction argument on $j$. Now, by applying the functor $\operatorname{Hom}_{R}(-, M)$ (resp. $-\otimes_{R} M$) to the above exact sequences, the assertion follows.
\end{proof}


\begin{definition}\label{2-1}
 
{\rm The \textit{$fp_n$-injective dimension}  of an $S$-module $M$ is defined that ${fp_n}.{\rm id}_{S}(M)\leq k$ if and only if there exists an exact sequence
\[0\longrightarrow M\longrightarrow I_0 \longrightarrow I_1\longrightarrow \cdots\longrightarrow I_{k- 1}\longrightarrow I_k\longrightarrow 0\]
of $S$-modules with each $I_i\in fp_nI(S)$ for all $0\leq i\leq k$. Also, the \textit{$fp_n$-flat dimension}  of an $R$-module $N$ is defined that ${fp_n}.{\rm fd}_{R}(N)\leq k$  if and only if there exists an exact sequence
\[0\longrightarrow F_k\longrightarrow F_{k- 1}\longrightarrow \cdots\longrightarrow F_1\longrightarrow F_0\longrightarrow N\longrightarrow 0\]
of $R$-modules with each $F_i\in fp_nF(R)$  for all $0\leq i\leq k$.}
\end{definition}

 It is clear that   ${fp_n}.{id}_{S}(M)\leq 0$ if and only if $M$ is  an $fp_n$-injective $S$-module, and  ${fp_n}.{\rm fd}_{R}(N)\leq 0$  if and only if $N$ is   an $fp_n$-flat  $R$-module. \\
For convenience, we set
\begin{itemize}
 \item $fp_{n}I(S)_{\leq k}=$ the class of $S$-modules of $fp_n$-injective dimension at most $k$.
\item  $fp_{n}F(R)_{\leq k}=$ the class of $R$-modules of $fp_n$-flat dimension at most $k$.
\end{itemize}

In the next lemma, we show that the Bass class $\mathcal{B}_C(S)$ contains all  $S$-modules with  finite $fp_n$-injective dimension and the Auslander class $\mathcal{A}_C(R)$ contains all  $R$-modules with finite $fp_n$-flat dimension.

\begin{lemma}\label{3-5}
Let $C= {}_SC_R$ be a faithfully semidualizing bimodule. Then the following assertions hold true:
\begin{itemize}
\item[\emph{(i)}] $fp_nI(S)_{\leq k}\subseteq \mathcal{B}_{C}(S)$;
\item[\emph{(ii)}] $fp_nF(R)_{\leq k}\subseteq \mathcal{A}_{C}(R)$.
\end{itemize}
\end{lemma}

\begin{proof}
(i). First, we show that for $k=0$, $fp_nI(S)\subseteq \mathcal{WI}(S)$. Cosider $\mathcal{Y}$-finitely presented $\mathcal{Y}=\cdots\longrightarrow F_j\longrightarrow F_{j- 1}\longrightarrow\cdots\longrightarrow F_1\longrightarrow F_0\longrightarrow U\longrightarrow 0$ of $S$-modules. Then there is an exact sequence $0\longrightarrow K_{0}\longrightarrow F_0\longrightarrow U\longrightarrow 0,$ where $K_{0}$, $F_{0}$ and $U$ are finitely $n$-finitely presented. So if $X\in fp_nI(S)$, then by Lemma \ref{3-4}, ${\rm Hom}_{S}(F_{0},X)\longrightarrow {\rm Hom}_{S}(K_{0},X)\longrightarrow 0$ is exact. Hence by Lemma \ref{3-3}, ${\rm Ext}_{S}^{1}(U,X)=0$, and then $X\in\mathcal{WI}(S)$. Consequently, $fp_nI(S)\subseteq \mathcal{B}_{C}(S)$ from \cite[Theorem 2.2]{Z.TT}. So for $M\in fp_nI(S)_{\leq k}$ there exists an exact sequence \[0\longrightarrow M\longrightarrow X_0 \longrightarrow X_1\longrightarrow \cdots\longrightarrow X_{k- 1}\longrightarrow X_k\longrightarrow 0\] of $S$-modules, where each $X_{i}\in\mathcal{B}_{C}(S)$ for all $0\leq i\leq k$. Then by \cite[Corollary 6.3]{HW}, we deduec that $M\in \mathcal{B}_{C}(S)$.

(ii). Let $N\in fp_nF(R)_{\leq k}$. Then there is an exact sequence \[0\longrightarrow F_k\longrightarrow F_{k- 1}\longrightarrow \cdots\longrightarrow F_1\longrightarrow F_0\longrightarrow N\longrightarrow 0\]
of $R$-modules with each $F_i\in fp_nF(R)$  for all $0\leq i\leq k$. Then by \cite[Lemma 3.53]{Rot2}, \[0\longrightarrow N^{*}\longrightarrow F^{*}_0 \longrightarrow F^{*}_1\longrightarrow \cdots\longrightarrow F^{*}_{k- 1}\longrightarrow F^{*}_k\longrightarrow 0\] of $R^{op}$-modules, where each $F^{*}_{i}\in fp_nI(R^{op})$ by \cite[Proposition 2.4(2)]{JZW}. By (i), $F^{*}_{i}\in\mathcal{B}_{C}(R^{op})$, and then \cite[Corollary 6.3]{HW} and Proposition \ref{3-5-A-B}, we deduec that $N^{*}\in \mathcal{B}_{C}(R^{op})$ if and only if $N\in\mathcal{A}_{C}(R)$.
\end{proof}

\section{$C$-$fp_{n}$-injective and $C$-$fp_{n}$-flat modules}\label{3}





\begin{definition}\label{3-1}
\emph{An $R$-module is called \textit{$C$-$fp_n$-injective} if it has the form $\operatorname{Hom}_{S}(C, X)$ for some  $X\in fp_nI(S)$. An $S$-module is called  \textit{$C$-$fp_n$-flat} if it has the form $C\otimes_{R} Y$ for some $Y\in fp_nF(R)$. We denote the class of $C$-$fp_n$-injective $R$-modules by $C{fp}_nI(R)$ and the class of $C$-$fp_n$-flat $S$-modules by $C{fp}_nF(S)$. Therefore
\[C{fp}_nI(R)= \{\operatorname{Hom}_{S}(C, X ): X\in fp_nI(S)\}\]
and
\[C{fp}_nF(S)= \{C\otimes_{R} Y: Y\in fp_nF(R)\}.\]}
\end{definition}


\begin{remark}\label{3-2}
\begin{itemize}
\item[\emph{(i)}] Every $C$-$FP_n$-injective $($resp. $C$-$FP_n$-flat$)$ module is $C$-$fp_m$-injective $($resp. $C$-$fp_m$-flat$)$ module for any $m\geq n$ $($see \emph{\cite[Proposition 1.7(1)]{ARB}}$)$. But not conversely,  see Example \ref{2-1-2}.

\item[\emph{(ii)}] Over $n$-coherent rings, every $C$-$fp_n$-injective $($resp. $C$-$fp_n$-flat$)$ module is also $C$-$FP_m$-injective $($resp. $C$-$FP_m$-flat$)$ module for any $m\geq n$ $($see \emph{Proposition \ref{2-2}}$)$;

\item[\emph{(iii)}] Every $C$-$fp_n$-injective $($resp. $C$-$fp_n$-flat$)$ module is $C$-$fp_m$-injective $($resp. $C$-$fp_m$-flat$)$ for all $m\geq n$, and so we have
\[C{fp}_1I(R)\subseteq C{fp}_2I(R)\subseteq \cdots\subseteq C{fp}_nI(R)\subseteq C{fp}_{n+1}I(R)\subseteq \cdots\]
and
\[C{fp}_1F(S)\subseteq C{fp}_2F(S)\subseteq \cdots\subseteq C{fp}_nF(S)\subseteq C{fp}_{n+1}F(S)\subseteq \cdots.\]

\item[\emph{(iv)}]  An $R$-module $M$ (resp. $S$-module) is $C$-$fp_{\infty}$-injective (resp.  $C$-$fp_{\infty}$-flat) if and only if weak injective (resp. weak flat). 
\end{itemize}
\end{remark}
Recall that a ring $R$ is said to be an \textit{$(n, 0)$-ring} or \textit{$n$-regular ring} if every finitely $n$-presented $R$-module is projective (see \cite{L.g,.TZ}).
\begin{example}\label{2-1-2}
{\rm Let $K$ be a field, $E$ a $K$-vector space with infinite rank, and $A$ a Noetherian ring of global dimension $0$. Set $B= K\ltimes E$ the trivial extension of $K$ by $E$ and $R= A\times B$ the direct product of $A$ and $B$. By \cite[Theorem 3.4(3)]{L.g}, $R$ is a $(2, 0)$-ring which is not a $(1, 0)$-ring. Thus, for every $R$-module $M$ and every finitely $2$-presented $R$-module $L$, $\operatorname{Ext}_{R}^{1}(L, M)= 0$ (resp. $\operatorname{Tor}_{1}^{R}(L, M)= 0$) . Hence every $R$-module is $FP_2$-injective (resp. $FP_2$-flat), and so every $R$-module is $fp_2$-injective (resp. $fp_2$-flat). On the other hand, there exists an $R$-module which is not $FP_1$-injective (resp. $FP_{1}$-flat), since  if every $R$-module is $FP_1$-injective (resp. $FP_{1}$-flat), \cite[Theorem 3.9]{.TZ} implies that  $R$ is  $(1, 0)$-ring, contradiction. Therefore,  if $C=R=S$, then every $R$-module is $C$-$fp_2$-injective and $C$-$fp_2$-flat, and there exists an $R$-module which is not $C$-$FP_1$-injective (resp. $C$-$FP_{1}$-flat).}
\end{example}

\begin{definition}\label{4-1}
\emph{Let $C= {}_SC_R$ be a faithfully semidualizing bimodule. The \textit{$C$-$fp_n$-injective dimension}  of an $R$-module $M$ is defined that ${Cfp_n}.{\rm id}_{R}(M)\leq k$ if and only if there exists an exact sequence
\[0\longrightarrow M\longrightarrow I_0 \longrightarrow I_1\longrightarrow \cdots\longrightarrow I_{k- 1}\longrightarrow I_k\longrightarrow 0\]
of $R$-modules with each $I_i\in Cfp_nI(R)$ for all $0\leq i\leq k$. Also, the \textit{$C$-$fp_n$-flat dimension}  of an $S$-module $N$ is defined that ${Cfp_n}.{\rm fd}_{S}(N)\leq k$  if and only if there exists an exact sequence
\[0\longrightarrow F_k\longrightarrow F_{k- 1}\longrightarrow \cdots\longrightarrow F_1\longrightarrow F_0\longrightarrow N\longrightarrow 0\]
of $S$-modules with each $F_i\in Cfp_nF(S)$  for all $0\leq i\leq k$.}
\end{definition}

It is clear that ${Cfp_n}.{\rm id}_{R}(M)\leq 0$  if and only if $M$ is a $C$-$fp_n$-injective $R$-module, and ${Cfp_n}.{\rm fd}_{S}(N)\leq 0$  if and only if $N$ is a $C$-$fp_n$-flat $S$-module.

For convenience, we set
\begin{itemize}
\item $Cfp_{n}I(R)_{\leq k}=$ the class of $R$-modules of $C$-$fp_n$-injective dimension at most $k$. 
\item $Cfp_{n}F(S)_{\leq k}=$ the class of $S$-modules of $C$-$fp_n$-flat dimension at most $k$. 
\end{itemize}


The following lemma is needed in the proof of the first main result of this section.

\begin{lemma}\label{4-2}
 Then the following assertions hold true:
\begin{itemize}
\item[\emph{(i)}] $Cfp_{n}I(R)_{\leq k}\subseteq \mathcal{A}_{C}(R)$;
\item[\emph{(ii)}] $Cfp_{n}F(S)_{\leq k}\subseteq \mathcal{B}_{C}(S)$.
\end{itemize}
\end{lemma}

\begin{proof}
(i).  Assume that $N\in Cfp_nI(R)$. Then $N= \operatorname{Hom}_{S}(C, X )$ for some $X\in fp_nI(S)$. By  Lemma \ref{3-5}(i), $X\in \mathcal{B}_{C}(S)$ and so $N\in \mathcal{A}_{C}(R)$ from \cite[Lemma 2.9(1)]{Z.TT}. Now, if $M\in Cfp_{n}I(R)_{\leq k}$, then there exists an exact sequence
\[0\longrightarrow M\longrightarrow I_0 \longrightarrow I_1\longrightarrow \cdots\longrightarrow I_{k- 1}\longrightarrow I_k\longrightarrow 0\]
of $R$-modules with each $I_i\in Cfp_nI(R)$  for all $0\leq i\leq k$, and also any $I_i\in \mathcal{A}_{C}(R)$. Hence by \cite[Corollary 6.3]{HW}, $M\in \mathcal{A}_{C}(R)$. 

(ii). This is similar to the first part.
\end{proof}



In the following, we investigate Foxby equivalence relative to the classes $Cfp_{n}I(R)$ and $Cfp_nF(S)$ as a generalization of Foxby equivalence relative to the classes $\mathcal{FI}^{n}_C(R)$ and $\mathcal{FF}^{n}_C(S)$ in \cite{WGZ}. 
\begin{proposition}\label{4-3}
 Then we have the following equivalences of categories:
\begin{itemize}
\item[\emph{(i)}]
$\xymatrix@C=3cm{Cfp_nI(R)_{\leq k}\ar@<1ex>[r]^{C\otimes_{R} -}_{\sim}&\ar@<1ex>[l]^{\operatorname{Hom}_{S}(C, -)}fp_{n}I(S)_{\leq k}};$
\item[\emph{(ii)}]
$\xymatrix@C=3cm{fp_nF(R)_{\leq k}\ar@<1ex>[r]^{C\otimes_{R} -}_{\sim}&\ar@<1ex>[l]^{\operatorname{Hom}_{S}(C, -)}Cfp_{n}F(S)_{\leq k}}.$
\end{itemize}
\end{proposition}

\begin{proof}
(i). Let $M\in Cfp_{n}I(R)_{\leq k}$. There exists an exact sequence
\[0\longrightarrow M\longrightarrow I_0 \longrightarrow I_1\longrightarrow \cdots\longrightarrow I_{k- 1}\longrightarrow I_k\longrightarrow 0\]
of $R$-modules with each $I_i\in Cfp_nI(R)$ for all $0\leq i\leq k$. Thus, $I_i= \operatorname{Hom}_{S}(C, X )$ for some $X\in fp_nI(S)$. By  Lemma \ref{3-5}(i), $X\in \mathcal{B}_{C}(S)$, and then $C\otimes_{R}\operatorname{Hom}_{S}(C, X )\cong X$. So  $C\otimes_{R} I_i\in fp_nI(S)$ and also by Lemma \ref{4-2}(i), $I_i\in \mathcal{A}_{C}(R)$, and so $\operatorname{Tor}_{j}^{R}(C, I_i)= 0$ for all $j\geq 1$. By Lemma \ref{4-2}(i), $M\in \mathcal{A}_{C}(R)$ and hence $\operatorname{Tor}_{j}^{R}(C, M)= 0$ for all $j\geq 1$. Therefore, by applying the functor $C\otimes_{R} -$ to the above exact sequence, we obtain the exact sequence
\[0\longrightarrow C\otimes_{R} M\longrightarrow C\otimes_{R} I_0 \longrightarrow C\otimes_{R} I_1\longrightarrow \cdots\longrightarrow C\otimes_{R} I_{k- 1}\longrightarrow C\otimes_{R} I_k\longrightarrow 0\]
of $S$-modules which shows that $C\otimes_{R} M\in fp_{n}I(S)_{\leq k}$.
Now, let $N\in fp_{n}I(S)_{\leq k}$. There exists an exact sequence
\[0\longrightarrow N\longrightarrow I'_0 \longrightarrow I'_1\longrightarrow \cdots\longrightarrow I'_{k- 1}\longrightarrow I'_k\longrightarrow 0\]
of $S$-modules with each $I'_i\in fp_nI(S)$ for all $0\leq i\leq k$. For all $0\leq i\leq k$, from Lemma \ref{3-5}(i), $I'_i\in \mathcal{B}_{C}(S)$, and so $\operatorname{Ext}_{S}^{j}(C, I'_i)= 0$ for all $j\geq 1$. Also, by Lemma \ref{3-5}(i), $N\in \mathcal{B}_{C}(S)$ and hence $\operatorname{Ext}_{S}^{j}(C, N)= 0$ for all $j\geq 1$. Therefore, by applying the functor $\operatorname{Hom}_{S}(C, -)$ to the above exact sequence, we obtain the exact sequence
\[0\rightarrow \operatorname{Hom}_{S}(C, N)\rightarrow \operatorname{Hom}_{S}(C, I'_0) \rightarrow \operatorname{Hom}_{S}(C, I'_1)\rightarrow \cdots\rightarrow \operatorname{Hom}_{S}(C, I'_{k- 1})\rightarrow \operatorname{Hom}_{S}(C, I'_k)\rightarrow 0\]
of $R$-modules which shows that $\operatorname{Hom}_{S}(C, N)\in Cfp_nI(R)_{\leq k}$. Note that, if $M\in Cfp_{n}I(R)_{\leq k}$, then by Lemma \ref{4-2}(iii), $M\in \mathcal{A}_{C}(R)$, and if $N\in fp_{n}I(S)_{\leq k}$, then from Lemma \ref{3-5}(i), $N\in \mathcal{B}_{C}(S)$. Hence we have the natural isomorphisms $M\cong \operatorname{Hom}_{S}(C, C\otimes_{R} M)$ and $C\otimes_{R} \operatorname{Hom}_{S}(C, N)\cong N$.

(ii). This is similar to that of (i).
\end{proof}
\begin{corollary}\label{4-4}
Let $C= {}_SC_R$ be a semidualizing bimodule. Then we have the following equivalences of categories:
\begin{itemize}
\item[\emph{(i)}]
$\xymatrix@C=3cm{Cfp_nI(R)\ar@<1ex>[r]^{C\otimes_{R} -}_{\sim}&\ar@<1ex>[l]^{\operatorname{Hom}_{S}(C, -)}fp_{n}I(S);}$
\item[\emph{(ii)}]
$\xymatrix@C=3cm{fp_nF(R)\ar@<1ex>[r]^{C\otimes_{R} -}_{\sim}&\ar@<1ex>[l]^{\operatorname{Hom}_{S}(C, -)}Cfp_{n}F(S).}$
\end{itemize}
\end{corollary}

\begin{proof}
Put $k= 0$ in Proposition \ref{4-3}.
\end{proof}


By using Lemma \ref{4-2}, Proposition \ref{4-3}, and Corollary \ref{4-4}, we get the first main result of this section.

\begin{theorem}\emph{(Foxby Equivalence)}\label{4-5}
 Then we have the following equivalences of categories:
\[\xymatrix@C=3cm{
fp_nF(R)\ar@<1ex>[r]^{C\otimes_{R} -}_{\sim}\ar@{^{(}->}[d]&\ar@<1ex>[l]^{\operatorname{Hom}_{S}(C, -)}Cfp_nF(S)\ar@{^{(}->}[d]\\
fp_nF(R)_{\leq k}\ar@<1ex>[r]^{C\otimes_{R} -}_{\sim}\ar@{^{(}->}[d]&\ar@<1ex>[l]^{\operatorname{Hom}_{S}(C, -)}Cfp_nF(S)_{\leq k}\ar@{^{(}->}[d]\\
\mathcal{A}_{C}(R)\ar@<1ex>[r]^{C\otimes_{R} -}_{\sim}& \ar@<1ex>[l]^{\operatorname{Hom}_{S}(C, -)}\mathcal{B}_{C}(S)\\
Cfp_{ n}I(R)_{\leq k}\ar@<1ex>[r]^{C\otimes_{R} -}_{\sim}\ar@{^{(}->}[u]&\ar@ <1ex>[l]^{\operatorname{Hom}_{S}(C, -)}fp_nI(S)_{\leq k}\ar@{^{(}->}[u]\\
Cfp_{ n}I(R)\ar@<1ex>[r]^{C\otimes_{R} -}_{\sim}\ar@{^{(}->}[u]&\ar@<1ex>[l]^{\operatorname{Hom}_{S}(C, -)}fp_nI(S)\ar@{^{(}->}[u].
}\]
\end{theorem}


We are now ready to state and prove the first main result of Theorem \ref{4-5} (Foxby Equivalence).
\begin{corollary}\label{4-6}
Let $M$ be an $R$-module, and $N$ an $S$-module. Then the following assertions hold true:
\begin{itemize}
\item[\emph{(i)}] $M\in Cfp_nI(R)_{\leq k}$ if and only if $M\in \mathcal{A}_{C}(R)$ and $C\otimes_{R} M\in fp_nI(S)_{\leq k}$;
\item[\emph{(ii)}] $N\in Cfp_nF(S)_{\leq k}$ if and only if $N\in \mathcal{B}_{C}(S)$ and $\operatorname{Hom}_{S}(C, N)\in fp_nF(R)_{\leq k}$.
\end{itemize}
\end{corollary}

\begin{proof}
(i). $(\Rightarrow)$ This follows from Lemma \ref{4-2}(i) and Theorem \ref{4-5}.

$(\Leftarrow)$ If $M\in \mathcal{A}_{C}(R)$ and $C\otimes_{R} M\in fp_nI(S)_{\leq k}$, then $M\cong \operatorname{Hom}_{S}(C, C\otimes_{R} M)$ and, by Theorem \ref{4-5}, $\operatorname{Hom}_S(C, C\otimes_{R} M)\in Cfp_nI(R)_{\leq k}$. Thus $M\in Cfp_nI(R)_{\leq k}$.

(ii). This is similar to the first part.
\end{proof}

\begin{corollary}\label{3-7}
Let $X$ be an $S$-module and $Y$ an $R$-module. Then the following statements hold true:
\begin{itemize}
\item[\emph{(i)}]  $\operatorname{Hom}_{S}(C, X)\in Cfp_nI(R)_{\leq k}$ if and only if $X\in fp_nI(S)_{\leq k}$.
\item[\emph{(ii)}] $C\otimes_{R} Y\in Cfp_nF(S)_{\leq k}$ if and only if $Y\in fp_nF(R)_{\leq k}$;
\end{itemize}
\end{corollary}

\begin{proof}
(i). Let $\operatorname{Hom}_{S}(C, X)\in Cfp_nI(R)_{\leq k}$. Then, by Corollary \ref{4-6}(i), $\operatorname{Hom}_{S}(C, X)\in \mathcal{A}_{C}(R)$. Therefore, from \cite[Lemma 2.9(1)]{Z.TT}, $X\in \mathcal{B}_{C}(S)$ and hence $C\otimes_{R} \operatorname{Hom}_{S}(C, X)\cong X$. Thus $X\in fp_nI(S)_{\leq k}$ by Theorem \ref{4-5}.

(ii). This is similar to that of (i).
\end{proof}


In the course of the remaining parts of the paper, we denote the character module of $M$ by $M^{*}:= \operatorname{Hom}_{\mathbb{Z}}(M, {\mathbb{Q}}/{\mathbb { Z}})$ \cite[Page 135]{Rot2}.

\begin{proposition}\label{3-8}
Let $M$ be an $R$-module and $N$ an $S$-module. Then the following statements hold:
\begin{itemize}
\item[\emph{(i)}] $M\in Cfp_nI(R)_{\leq k}$ if and only if $M^*\in Cfp_nF(R^{op})_{\leq k}$;
\item[\emph{(ii)}] $N\in Cfp_nF(S)_{\leq k}$ if and only if $N^*\in Cfp_nI(S^{op})_{\leq k}$.
\end{itemize}
\end{proposition}

\begin{proof}
(i).  Assume that $M\in Cfp_nI(R)_{\leq k}$. We proceed by induction on $k$. $(\Rightarrow)$ If $k=0$, then $M= \operatorname{Hom}_{S}(C, X )$ for some  $X\in fp_nI(S)$. From \cite[Proposition 2.4(1)]{JZW}, $X^*\in fp_nF(S^{op})$. Thus $M^*\in Cfp_nF(R^{op})$ because $M^{*}= \operatorname{Hom}_{S}(C, X)^{*}\cong C\otimes_{S^{op}} X^*$ by \cite[Lemma 3.55 and Proposition 2.56]{Rot2}. $(\Leftarrow)$ Assume that $M^*\in Cfp_nF(R^{op})$. Then, from Corollary \ref{4-6}(ii), $M^*\in \mathcal{B}_{C}(R^{op})$ and $\operatorname{Hom}_{R^{op}}(C, M^*)\in fp_nF(S^{op})$. Also, by \cite[Proposition 2.56 and Theorem 2.76]{Rot2}, $(C\otimes_{R} M)^*\cong \operatorname{Hom}_{R^{op}}(C, M^*)$ and so $C\otimes_{R} M\in fp_nI(S)$ from \cite[Proposition 2.4(1)]{JZW}. Since $M^*\in \mathcal{B}_{C}(R^{op})$, $M^*\cong C\otimes_{S^{op}} \operatorname{Hom}_{R^{op}}(C, M^*)\cong C\otimes_{S^{op}} (C\otimes_{R} M)^*\cong \operatorname{Hom}_{S}(C, C\otimes_{R} M)^*$ from \cite[Proposition 2.56, Theorem 2.76, and Lemma 3.55]{Rot2}. Hence $M\cong \operatorname{Hom}_{S}(C, C\otimes_{R} M)$ by \cite[Lemma 3.53]{Rot2}. Thus $M\in Cfp_nI(R)$.

Assume that $M\in Cfp_nI(R)_{\leq k}$. Then there exists an exact sequence
\[0\longrightarrow M\longrightarrow Y\longrightarrow L\longrightarrow0,\] where $Y\in Cfp_nI(R)$ and $L\in Cfp_nI(R)_{\leq k-1}$. Since $Y^*\in Cfp_nF(R^{op})$, and by \cite[Lemma 3.53]{Rot2}, \[0\longrightarrow L^*\longrightarrow Y^*\longrightarrow M^*\longrightarrow0\]  is an exact sequence, we deduce that $M\in Cfp_nI(R)_{\leq k}$ if and only if $L\in Cfp_nI(R)_{\leq k-1}$ if and only if $L^*\in Cfp_nF(R^{op})_{\leq k-1}$ if and only if $M^*\in Cfp_nF(R^{op})_{\leq k}$.

(ii). This is similar to the first part.
\end{proof}
\begin{corollary}\label{3-9}
Let $M$ be an $R$-module and $N$ an $S$-module. Then the following assertions hold:
\begin{itemize}
\item[\emph{(i)}] $M\in Cfp_nI(R)_{\leq k}$ if and only if $M^{**}\in Cfp_nI(R)_{\leq k}$;
\item[\emph{(ii)}] $N\in Cfp_nF(S)_{\leq k}$ if and only if $N^{**}\in Cfp_nF(S)_{\leq k}$.
\end{itemize}
\end{corollary}

\begin{proof}
This follows by Proposition \ref{3-8}.
\end{proof}


In the next proposition, we prove that the classes  $Cfp_nI(R)_{\leq k}$ and $Cfp_nF(S)_{\leq k}$ are closed under direct summands, direct products, and direct sums.

\begin{proposition}\label{3-10}
The following assertions hold:
\begin{itemize}
\item[\emph{(i)}] The class $Cfp_nI(R)_{\leq k}$ is closed under direct summands, direct products, and direct sums;
\item[\emph{(ii)}] The class $Cfp_nF(S)_{\leq k}$ is closed under direct summands, direct products, and direct sums.
\end{itemize}
\end{proposition}

\begin{proof}
(i). Let $M\in Cfp_nI(R)_{\leq k}$ and let $M'$ be a summand of $M$. Then, by Corollary \ref{4-6}(i), $M\in \mathcal{A}_{C}(R)$ and $C\otimes_{R} M\in fp_nI(S)_{\leq k}$,  and also there exists an $R$-module $M''$ such that $M\cong M'\oplus M''$. From \cite[Proposition 4.2(a)]{HW}, it follows that $M'\in \mathcal{A}_{C}(R)$. Also, by \cite[Theorem 2.65]{Rot2}, we have $C\otimes_{R} M\cong (C\otimes_{R} M')\oplus (C\otimes_{R} M'')$ which shows from \cite[Proposition 2.3(1)]{JZW} that $C\otimes_{R} M'\in fp_nI(S)_{\leq k}$. Thus $M'\in Cfp_nI(R)_{\leq k}$  by Corollary \ref{4-6}(i). \\
Now, let $\{M_j\}_{j\in J}$ be a family of  $R$-modules of $C$-$fp_n$-injective dimension at most $k$.  Then, by Corollary \ref{4-6}(i), $M_j\in \mathcal{A}_{C}(R)$ and $C\otimes_{R} M_j\in fp_nI(S)_{\leq k}$ for all $j\in J$. Hence, from \cite[Proposition 4.2(a)]{HW}, $\prod_{j\in J} M_j\in \mathcal{A}_{C}(R)$ (resp. $\bigoplus_{j\in J} M_j\in \mathcal{A}_{C}(R)$). On the other hand, there exists an exact sequence
\[0\longrightarrow C\otimes_{R}M_j\longrightarrow I_{0j} \longrightarrow I_{1j}\longrightarrow \cdots\longrightarrow I_{{k- 1}{j}}\longrightarrow I_{kj}\longrightarrow 0\]
of $S$-modules with each $I_{ij}\in  fp_nI(S)$ for all $0\leq i\leq k$. So we have the exact sequence
\[0\longrightarrow \prod_{j\in J}(C\otimes_{R}M_j)\longrightarrow \prod_{j\in J}I_{0j} \longrightarrow \prod_{j\in J}I_{1j}\longrightarrow \cdots\longrightarrow \prod_{j\in J}I_{{k- 1}{j}}\longrightarrow \prod_{j\in J}I_{kj}\longrightarrow 0\]
of $S$-modules, where by \cite[Proposition 2.3(1)]{JZW},  $\prod_{j\in J}I_{ij}\in  fp_nI(S)$ for all $0\leq i\leq k$, and so $\prod_{j\in J}(C\otimes_{R}M_j)\in fp_nI(S)_{\leq k}$. Similarly, 
 $\bigoplus_{j\in J} (C\otimes_{R} M_j)\in fp_nI(S)_{\leq k}$.  Since $C$ is finitely presented, from \cite[Lemma 2.10(2)]{NDING} we have  $C\otimes_{R} (\prod_{j\in J} M_j)\cong \prod_{j\in J}(C\otimes_{R}M_j)$, and then  $C\otimes_{R} (\prod_{j\in J} M_j)\in fp_nI(S)_{\leq k}$.  Also, $C\otimes_{R} (\bigoplus_{j\in J} M_j)\in fp_nI(S)$  by \cite[Theorem 2.65]{Rot2}. Thus $\prod_{j\in J} M_j\in Cfp_nI(R)_{\leq k}$ (resp. $\bigoplus_{j\in J} M_j\in Cfp_nI(R)_{\leq k}$) by Corollary \ref{4-6}(i).

(ii). By using \cite[Theorem 2.30 and Corollary 2.32]{Rot2} and \cite[Lemma 2.9]{NDING}, the proof is similar to that of (i).
\end{proof}
Let $\mathcal{F}$ be a class of $R$-modules and let $M$ be an $R$-module. A morphism $f: F\longrightarrow M$ (resp. $f: M\longrightarrow F$) with $F\in \mathcal{F}$ is called an \textit{$\mathcal{F}$-precover} (resp. \textit{$\mathcal{F}$-preenvelope}) of $M$ when $\operatorname{Hom}_{R}(F', F)\longrightarrow \operatorname{Hom}_{R}(F', M)\longrightarrow 0$ (resp. $\operatorname{Hom}_{R}(F, F')\longrightarrow \operatorname{Hom}_{R}(M, F')\longrightarrow 0$) is exact for all $F'\in \mathcal{F}$. Assume that $f: F\longrightarrow M$ (resp. $f: M\longrightarrow F$) is an $\mathcal{F}$-precover (resp. $\mathcal{F}$-preenvelope) of $M$. Then $f$ is called an \textit{$\mathcal{F}$-cover} (resp. \textit{$\mathcal{F}$-envelope}) of $M$ if every morphism $g: F\longrightarrow F$ such that $fg = f$ (resp. $gf = f$) is an isomorphism. The class $\mathcal{F}$ is called \textit{$($pre$)$covering} (resp. \textit{$($pre$)$enveloping}) if each $R$-module has an $\mathcal{F}$-(pre)cover (resp. $\mathcal{F}$-(pre)envelope) (see \cite[Definitions 5.1.1 and 6.1.1]{EM}).

A \textit{duality pair} over $R$ is a pair $(\mathcal{M}, \mathcal{N})$, where $\mathcal{M}$ is a class of $R$-modules and $\mathcal{N}$ is a class of $R^{op}$-modules, subject to the following conditions:
\begin{itemize}
\item[(i)] For an $R$-module $M$, one has $M\in \mathcal{M}$ if and only if $M^{*}\in \mathcal{N}$;
\item[(ii)] $\mathcal{N}$ is closed under direct summands and finite direct sums.
\end{itemize}
A duality pair $(\mathcal{M}, \mathcal{N})$ is called \textit{$($co$)$product-closed} if the class $\mathcal{M}$ is closed under (co)products in the category of all $R$-modules (see \cite[Definition 2.1]{HJ}).

\begin{corollary}\label{3-11}
$(Cfp_{n}I(R)_{\leq k}, Cfp_nF(R^{op})_{\leq k})$ and $(Cfp_{n}F(S)_{\leq k}, Cfp_nI(S^{op})_{\leq k})$ are duality pairs.
\end{corollary}

\begin{proof}
By Proposition \ref{3-8}, an $R$-module $M$ (resp. $S$-module $N$) is in $Cfp_{n}I(R)_{\leq k}$ (resp. $Cfp_{n}F(S)_{\leq k}$) if and only if $M^{*}$ (resp. $N^{*}$) is in $Cfp_nF(R^{op})_{\leq k})$ (resp. $Cfp_nI(S^{op})_{\leq k}$). Also, from Proposition \ref{3-10}, $Cfp_nF(R^{op})_{\leq k}$ (resp. $Cfp_nI(S^{op})_{\leq k}$) is closed under direct summands and direct sums. Thus the assertions follow.
\end{proof}


Assume that $M'$ is an $R$-submodule of $M$. We say that $M'$ is a \textit{pure submodule} of $M$, $M/M'$ is a \textit{pure quotient} of $M$, and $M$ is a \textit{pure extension} of $M'$ and $M/M'$ if
\[0\longrightarrow A\otimes_R M'\longrightarrow A\otimes_R M\longrightarrow A\otimes_R M/M'\longrightarrow 0\]
is an exact sequence for all $R^{op}$-modules $A$, equivalently, if
\[0\longrightarrow \operatorname{Hom}_{R}(B, M')\longrightarrow \operatorname{Hom}_{R}(B, M)\longrightarrow \operatorname{Hom}_{R}(B, M/M')\longrightarrow 0\]
is an exact sequence for all finitely $1$-presented $R$-modules $B$ \cite[Definition 5.3.6]{EM}.

Wei and Zhang proved in \cite[Proposition 2.4(2)]{JZW} that the classes $fp_nI(R)$ and $fp_nF(R)$ are closed under pure submodules and pure quotients. The next corollary shows that the classes $Cfp_nI(R)$ and $Cfp_nF(S)$ are also closed under pure submodules, pure quotients, and pure extensions.

\begin{corollary}\label{3-12}
Let $M'$ be a pure submodule of $R$-module $M$ and let $N'$ be a pure submodule of $S$-module $N$. Then the following statements hold true:
\begin{itemize}
\item[\emph{(i)}] $M\in Cfp_nI(R)$ if and only if $M'\in Cfp_nI(R)$ and $M/M'\in Cfp_nI(R)$;
\item[\emph{(ii)}] $N\in Cfp_nF(S)$ if and only if $N'\in Cfp_nF(S)$ and $N/N'\in Cfp_nF(S)$.
\end{itemize}
\end{corollary}

\begin{proof}
The assertion follows by Corollary \ref{3-11} and \cite[Theorem 3.1]{HJ}.
\end{proof}
In the second main result of this section, by the use of duality pairs, we show that $Cfp_{n}I(R)_{\leq k}$ and $Cfp_{n}F(S)_{\leq k}$ are preenveloping and covering.

\begin{theorem}\label{3-13}
The classes $Cfp_{n}I(R)_{\leq k}$ and $Cfp_{n}F(S)_{\leq k}$ are preenveloping and covering.
\end{theorem}

\begin{proof}
By Corollary \ref{3-11}, $(Cfp_{n}I(R)_{\leq k}, Cfp_nF(R^{op})_{\leq k})$ and $(Cfp_{n}F(S)_{\leq k}, Cfp_nI(S^{op})_{\leq k})$ are duality pairs. Also, from Proposition \ref{3-10}, the classes $Cfp_{n}I(R)_{\leq k}$ and $Cfp_{n}F(S)_{\leq k}$ are closed under direct products and direct sums. Therefore, from \cite[Theorem 3.1]{HJ}, the classes $Cfp_{n}I(R)_{\leq k}$ and $Cfp_{n}F(S)_{\leq k}$ are preenveloping and covering.
\end{proof}
\section{  $C$-$fp_n$-injective and $C$-$fp_n$-flat dimension of modules with respect to change of rings}\label{4}



We assume $S\leq R$ is a unitary ring extension. The ring $S$ is called
right $R$-projective, \cite{WX,TX} in case, for any right $S$-module $M_S$ with an $S$ module $N_S,$ $N_{R}\mid M_{R}$ implies $N_{S}\mid M_{S}$,
where $N\mid M$ means $N$ is a direct summand of $M$. $S$
is called a finite normalizing extension of $R$ if there exist elements $a_{1},\cdots,a_{n}\in S$
such that $a_{1}=1$, $S=Ra_{1}+\cdots+Ra_{n}$.  A finite normalizing extension $S\leq R$
is called an almost excellent extension in case $_RS$ is flat, $S_R$ is projective,
and the ring $S$ is right $R$-projective. An almost excellent extension $S\leq R$ is an excellent extension  in case both $_RS$ and $S_R$ are free modules with a common basis $\{a_{1},\cdots,a_{n}\}$.

In this section, we investigat modules of $C$-$fp_n$-injective dimension at most $k$ and also, modules of  $C$-flat dimension at most $k$ under an almost excellent extension of rings, where $C$ is a faithfully semidualizing $R$-module.

\begin{lemma}\label{1.p}
Let $S\geq R$ be an almost excellent extension. Then the following assertions hold:
\begin{itemize}
\item[\emph{(i)}]
 If $X\in fp_{n}I(R)_{\leq k}$, then ${\rm Hom}_{R}(S,X)\in fp_{n}I(S)_{\leq k}$.
\item[\emph{(ii)}]
 If $X\in fp_{n}F(R)_{\leq k}$, then $(S\otimes_{R}X)\in fp_{n}F(S)_{\leq k}$.
\end{itemize}
\end{lemma} 
\begin{proof}
(i). Consider, the exact sequence $0\longrightarrow  K\longrightarrow L$, where $K$ and $L$ are finitely $n$-presented $S$-modules. By \cite[Theorem 5]{wu}, $K$ and $L$ are finitely $n$-presented $R$-modules. If $k=0$, then $X\in fp_{n}I(R)$. We show that ${\rm Hom}_{R}(S,X)\in fp_{n}I(S)$. We have the commutative diagram
\[
\xymatrix{
&\operatorname{ Hom}_{S}(L,{\rm Hom}_{R}(S,X))\ar[r]\ar[d]^{\cong}&\operatorname{Hom}_{S}(K,{\rm Hom}_{R}(S,X))\ar[d]^{\cong}\\
&\operatorname{ Hom}_{R}(L,X)\ar[r]&\operatorname{Hom}_{R}(K,X)\ar[r]&0,
}
\] 
and so, the sequence $\operatorname{ Hom}_{S}(L,{\rm Hom}_{R}(S,X))\longrightarrow\operatorname{Hom}_{S}(K,{\rm Hom}_{R}(S,X))\longrightarrow 0$ is exact and hence ${\rm Hom}_{R}(S,X)\in fp_{n}I(S)$.

Now, let $X\in fp_{n}I(R)_{\leq k}$. Then there exists an exact sequence
\[0\longrightarrow X\longrightarrow X_0 \longrightarrow X_1\longrightarrow \cdots\longrightarrow X_k\longrightarrow 0\]
of $R$-modules with each $X_i\in fp_nI(R)$ for all $0\leq i\leq k$. Since $S_R$ is projective, there exists an exact sequence
\[0\longrightarrow \operatorname{ Hom}_{R}(S,X)\longrightarrow \operatorname{ Hom}_{R}(S,X_0) \longrightarrow \operatorname{ Hom}_{R}(S,X_1)\longrightarrow \cdots\longrightarrow \operatorname{ Hom}_{R}(S,X_k)\longrightarrow 0\]
of $S$-modules with each $\operatorname{ Hom}_{R}(S,X_i)\in fp_nI(S)$ for all $0\leq i\leq k$. Thus, ${\rm Hom}_{R}(S,X)\in fp_{n}I(S)_{\leq k}$.

(ii). By Definition \ref{2-1} and \cite[Proposition 2.4(1)]{JZW},   it follows that for an  $R$-module $Y$, $Y\in fp_{n}I(R)_{\leq k}$ if and only if $Y^{*}\in fp_{n}F(R^{op})_{\leq k}$ and $Y\in fp_{n}F(R)_{\leq k}$ if and only if $Y^{*}\in fp_{n}I(R^{op})_{\leq k}$. So if $X\in fp_{n}F(R)_{\leq k}$, then $X^{*}\in fp_{n}I(R^{op})_{\leq k}$. Hence by (1) and  \cite[Proposition 2.56 and Theorem 2.76]{Rot2},  $(S\otimes_{R}X)^{*}\cong{\rm Hom}_{R}(S,X^{*})\in fp_{n}I(S^{op})_{\leq k}$, and then $(S\otimes_{R}X)\in fp_{n}F(S)_{\leq k}$. 
\end{proof}

\begin{lemma}\label{1.p1}
Let $S\geq R$ be an almost excellent extension and $C$ a (faithfully) semidualizing $R$-module. Then $C\otimes_{R}S$ is a faithfully semidualizing $S$-module.
\end{lemma} 
\begin{proof}
Let $C$ a faithfully semidualizing $R$-module. Then by \cite[Lemma 3.4]{JZW}, $C\otimes_{R}S$ is a semidualizing $S$-module. Let ${\rm Hom}_{S}(C\otimes_{R}S, N)=0$ for a $S$-module $N$. Then $0={\rm Hom}_{S}(C\otimes_{R}S, N)\cong {\rm Hom}_{R}(C, {\rm Hom}_{S}(C, N))\cong{\rm Hom}_{R}(C, N)$, and so $N=0$.
\end{proof}
\begin{proposition}\label{1.p1pp}
Let $S\geq R$ be an almost excellent extension . Then the following assertions hold true:
\begin{itemize}
\item[\emph{(i)}]
 If $M\in Cfp_{n}I(R)_{\leq k}$, then 
 ${\rm Hom}_{R}(S,M)\in (C\otimes_{R}S)fp_{n}I(S)_{\leq k}$;
\item[\emph{(ii)}]
  If $M\in Cfp_{n}F(R)_{\leq k}$, then 
 $(S\otimes_{R}M)\in (C\otimes_{R}S)fp_{n}F(S)_{\leq k}$.
\end{itemize}
\end{proposition} 
\begin{proof}
(i).  Let $M\in Cfp_{n}I(R)_{\leq k}$. If $k=0$, then  $M={\rm Hom}_R(C,X)$ for some $X\in fp_{n}I(R)$.  We have 
\[\begin{array}{lllllll}
 {\rm Hom}_{R}(S,M)
&\cong {\rm Hom}_{R}(S,{\rm Hom}_R(C,X))\\
&\cong {\rm Hom}_{R}(C\otimes_{R}S,X)\\
&\cong {\rm Hom}_{R}(C\otimes_{R}S\otimes_{S}S,X)\\
&\cong {\rm Hom}_{S}(C\otimes_{R}S,{\rm Hom}(S,X)).\\
\end{array}\]
 Since by  Lemma \ref{1.p}, ${\rm Hom}_{R}(S,X)\in fp_{n}I(S)$ and by Lemma \ref{1.p1}, $C\otimes_{R}S$ is semidualizing $S$-module, we deduce that ${\rm Hom}_{S}(C\otimes_{R}S,{\rm Hom}(S,X))\in(C\otimes_{R}S)fp_{n}I(S)$. So, it follows that 
${\rm Hom}_{R}(S,M)\in (C\otimes_{R}S)fp_{n}I(S)$. 

(ii). This is similar to that of (i).
\end{proof}
 In the following, we give equivalent conditions with modules of $C$-$fp_n$-injective dimension at most $k$ and also, modules of  $C$-$fp_n$-flat dimension at most $k$ under almost excellent extension of rings.

\begin{proposition}\label{1.p10p}
Let $S\geq R$ be an almost excellent extension and $M$ an $S$-module. Then the following assertions are equivalent:
\begin{itemize}
\item[\emph{(i)}]
 $M\in Cfp_{n}I(R)_{\leq k}$;
\item[\emph{(ii)}]
 ${\rm Hom}_{R}(S,M)\in (C\otimes_{R}S)fp_{n}I(S)_{\leq k}$;
\item[\emph{(iii)}]
 $M\in (C\otimes_{R}S)fp_{n}I(S)_{\leq k}$.
\end{itemize}
\end{proposition} 
\begin{proof}
(i)$\Longrightarrow$(ii).  Let $M\in Cfp_{n}I(R)_{\leq k}$. Then by Proposition \ref{1.p1pp}(1), 
 ${\rm Hom}_{R}(S,M)\in (C\otimes_{R}S)fp_{n}I(S)_{\leq k}$. 

(ii)$\Longrightarrow$(iii). By \cite[Lemma 1.1]{WX}, $_SM$ is isomorphic to a direct summand of $S$-module ${\rm Hom}_{R}(S,M)$. Then by (2) and Proposition \ref{3-10}(1), $M\in (C\otimes_{R}S)fp_{n}I(S)_{\leq k}$.

(iii)$\Longrightarrow$ (i). Let $k=0$. Then $M\in (C\otimes_{R}S)fp_{n}I(S)$, and so $M={\rm Hom}_{S}(C\otimes_{R}S, X)$ for some $X\in fp_{n}I(S)$. We have  
$M={\rm Hom}_{S}(C\otimes_{R}S, X)\cong{\rm Hom}_{R}(C, {\rm Hom}_{S} (S,X))\cong{\rm Hom}_{R}(C, X)$. We show that $X\in fp_{n}I(R)$. 
Let $0\longrightarrow  K\longrightarrow L$ 
be an exact sequence of $R$-modules, where $K$ and $L$ are finitely $n$-presented $R$-modules.  Since $S$ is flat $R$-module, we have that 

$0\longrightarrow  K\otimes_{R}S\longrightarrow L\otimes_{R}S$ is an exact sequence of $S$-modules, where $K\otimes_{R}S$ and $K\otimes_{R}S$ are finitely $n$-presented $S$-modules by \cite[Lemma 4]{wu}. We have the commutative diagram
\[
\xymatrix{
&\operatorname{ Hom}_{S}(L\otimes_{R}S,X)\ar[r]\ar[d]^{\cong}&\operatorname{Hom}_{S}(K\otimes_{R}S, X)\ar[r]\ar[d]^{\cong}&0\\
&\operatorname{ Hom}_{R}(L,X)\ar[r]&\operatorname{Hom}_{R}(K,X).&
}
\] 
So, the sequence $\operatorname{ Hom}_{R}(L,X)\longrightarrow\operatorname{Hom}_{R}(K,X)\longrightarrow 0$ is exact, and then $X\in fp_{n}I(R)$.  Therefore, we get that $M\in Cfp_{n}I(R)$.
 Also, if $M\in (C\otimes_{R}S)fp_{n}I(S)_{\leq k}$, it simply follows that $M\in Cfp_{n}I(R)_{\leq k}$.
\end{proof}
\begin{proposition}\label{1.p10}
Let $S\geq R$ be an almost excellent extension and $M$ an $S$-module. Then the following assertions are equivalent:
\begin{itemize}
\item[\emph{(i)}]
 $M\in Cfp_{n}F(R)_{\leq k}$;
\item[\emph{(ii)}]
 $(S\otimes_{R}M)\in (C\otimes_{R}S)fp_{n}F(S)_{\leq k}$;
\item[\emph{(iii)}]
 $M\in (C\otimes_{R}S)fp_{n}F(S)_{\leq k}$.
\end{itemize}
\end{proposition} 
\begin{proof}
 By Propositions \ref{1.p10p} and \ref{3-8} and \cite[Proposition 2.56 and Theorem 2.76]{Rot2}, $M\in Cfp_{n}F(R)_{\leq k}$ if and only if $M^*\in Cfp_{n}I(R^{op})_{\leq k}$ if and only if  ${\rm Hom}_{R}(S,M^{*})\in (C\otimes_{R}S)fp_{n}I(S^{op})_{\leq k}$ if and only if $(S\otimes_{R}M)^{*}\in (C\otimes_{R}S)fp_{n}I(S^{op})_{\leq k}$ if and only if $(S\otimes_{R}M)\in (C\otimes_{R}S)fp_{n}F(S)_{\leq k}$. ALso, $M\in Cfp_{n}F(R)_{\leq k}$ if and only if $M^*\in Cfp_{n}I(R^{op})_{\leq k}$ if and only if $M^{*}\in (C\otimes_{R}S)fp_{n}I(S^{op})_{\leq k}$ if and only if $M\in (C\otimes_{R}S)fp_{n}F(S)_{\leq k}$ .
\end{proof}
\begin{corollary}\label{1.p1000}
Let $S\geq R$ be an almost excellent extension and $R$ an $n$-coherent ring. Then the following assertions hold true:
\begin{itemize}
\item[\emph{(i)}]
 The class $(C\otimes_{R}S)fp_{n}I(S)_{\leq k}$ is closed under extentions and  cokernels of
monomorphisms.
\item[\emph{(ii)}]
The class $(C\otimes_{R}S)fp_{n}F(S)_{\leq k}$ is closed under extentions and kernels of epimorphisms. 
\end{itemize}
\end{corollary} 
\begin{proof}
(i). Consider the exact sequence $0\longrightarrow A\longrightarrow B\longrightarrow C\longrightarrow 0,$ of $S$-modules, where $A$ and $C$ are in  $(C\otimes_{R}S)fp_{n}I(S)_{\leq k}$.  Then by Proposition \ref{1.p10p}, $A$ and $C$ are in  $Cfp_{n}I(R)_{\leq k}$. So by Remark \ref{3-2}(ii) and \cite[Theorem 4.9]{WGZ}, $B$ is in  $Cfp_{n}I(R)_{\leq k}$, and then $B$ is in $(C\otimes_{R}S)fp_{n}I(S)_{\leq k}$ from Proposition \ref{1.p10p}. Similarly, if $B$ and $C$ are in  $(C\otimes_{R}S)fp_{n}I(S)_{\leq k}$, then $A$ is in $(C\otimes_{R}S)fp_{n}I(S)_{\leq k}$.

(ii). This is similar to that of (i) by using Proposition \ref{1.p10} and \cite[Theorem 4.8]{WGZ}.
\end{proof}
\begin{theorem}\label{1.p124}
Let $S\geq R$ be an almost excellent extension. Then the class $(C\otimes_{R}S)fp_{n}I(S)_{\leq k}$ is preenveloping and precovering.
\end{theorem} 
\begin{proof}
Let $M$ is an $S$-module. We show that $M$ has a $(C\otimes_{R}S)fp_{n}I(S)_{\leq k}$-preenvelope. Since $M$ is an $R$-module, then by Theorem \ref{3-13},  
$M$ has a $Cfp_{n}I(R)_{\leq k}$-preenvelope. Let $R$-homomorphism $\alpha:M\longrightarrow N$ be a $Cfp_{n}I(R)_{\leq k}$-preenvelope of $M$.  Then by Proposition \ref{1.p1pp}(1),  ${\rm Hom}_{R}(S,N)\in (C\otimes_{R}S)fp_{n}I(S)_{\leq k}$. We  prove that $\alpha_{*}{\lambda}_{M}:M\longrightarrow{\rm Hom}_{R}(S,N)$ is a $(C\otimes_{R}S)fp_{n}I(S)_{\leq k}$-preenvelope of $S$-module $M$, where $\lambda_{M}:M\longrightarrow{\rm Hom}_{R}(S,M)$ and $\alpha_{*}:{\rm Hom}_{R}(S,M)\longrightarrow{\rm Hom}_{R}(S,N).$  If $L\in(C\otimes_{R}S)fp_{n}I(S)_{\leq k}$, and $\beta:M\longrightarrow L$ is an $S$-homomorphism, then by Proposition \ref{1.p10p}, $L\in Cfp_{n}I(R)_{\leq k}$, and so  there exists $R$-homomorphism $\gamma:N\longrightarrow L$ such that $\beta=\gamma\alpha$. Thus, we have the following commutative diagram:
\[
\xymatrix@C=3cm{
_SM\ar[d]^{\beta}\ar@<1ex>[r]^{\lambda_{M}}_{}&\ar@<1ex>[l]^{\pi_{M}}{\rm Hom}_{R}(S,M)\ar[r]^{\alpha_{*}}\ar[d]^{\beta_{*}}&{\rm Hom}_{R}(S,N)\ar[d]^{1}&\\
_SL\ar@<1ex>[r]^{\lambda_{L}}_{}&\ar@<1ex>[l]^{\pi_{L}}{\rm Hom}_{R}(S,L)&\ar^{\gamma_{*}}[l] {\rm Hom}_{R}(S,N)&.}
\]
So, we have $(\pi_{L}\gamma_{*})(\alpha_{*}\lambda_{M})=\pi_{L}(\gamma_{*}\alpha_{*})\lambda_{M}=\pi_{L}(\gamma\alpha)_{*}\lambda_{M}=\pi_{L}(\beta)_{*}\lambda_{M}=\pi_{L}\lambda_{L}\beta=\beta$. Therefore, we get that every $S$-module $M$ has a $(C\otimes_{R}S)fp_{n}I(S)_{\leq k}$-preenvelope. Similarly, it is proved that the class $(C\otimes_{R}S)fp_{n}I(S)_{\leq k}$ is  precovering.
\end{proof}
\begin{theorem}\label{1.p125}
Let $S\geq R$ be an almost excellent extension. Then the class $(C\otimes_{R}S)fp_{n}F(S)_{\leq k}$ is preenveloping and precovering.
\end{theorem} 
\begin{proof}
Let $M$ is an $S$-module. We show that $M$ has a $(C\otimes_{R}S)fp_{n}F(S)_{\leq k}$-preenvelope. Since $M$ is an $R$-module, then by Theorem \ref{3-13},  
$M$ has a $Cfp_{n}F(R)_{\leq k}$-preenvelope. Let $R$-homomorphism $\alpha:M\longrightarrow N$ be a $Cfp_{n}F(R)_{\leq k}$-preenvelope of $M$.  Then by Proposition \ref{1.p1pp}(2),  $(S\otimes_{R}N)\in (C\otimes_{R}S)fp_{n}F(S)_{\leq k}$. We  prove that $(S\otimes_{R}\alpha){\l}_{M}:M\longrightarrow S\otimes_{R}N$ is a $(C\otimes_{R}S)fp_{n}F(S)_{\leq k}$-preenvelope of $S$-module $M$, where $\l_{M}:M\longrightarrow(S\otimes_{R}M)$ and $S\otimes_{R}\alpha:S\otimes_{R}M\longrightarrow S\otimes_{R}N.$  If $L\in(C\otimes_{R}S)fp_{n}F(S)_{\leq k}$, and $\beta:M\longrightarrow L$ is an $S$-homomorphism, then by Proposition \ref{1.p10}, $L\in Cfp_{n}F(R)_{\leq k}$, and so  there exists $R$-homomorphism $\gamma:N\longrightarrow L$ such that $\beta=\gamma\alpha$. Thus, we have the following commutative diagram:
\[
\xymatrix@C=3cm{
_SM\ar[d]^{\beta}\ar@<1ex>[r]^{\l_{M}}_{}&\ar@<1ex>[l]^{\tau_{M}}S\otimes_{R}M\ar[r]^{S\otimes_{R}\alpha}\ar[d]^{S\otimes_{R}\beta}&S\otimes_{R}N\ar[d]^{1}&\\
_SL\ar@<1ex>[r]^{\l_{L}}_{}&\ar@<1ex>[l]^{\tau_{L}}S\otimes_{R}L&\ar^{S\otimes_{R}\gamma}[l] S\otimes_{R}N&.}
\]
Thus, we have $\tau_{L}(S\otimes_{R}\gamma)(S\otimes_{R}\alpha)\l_{M}=\tau_{L}(S\otimes_{R}\gamma\alpha)\l_{M}=\tau_{L}l_{L}\beta=\beta$, and so every $S$-module $M$ has a $(C\otimes_{R}S)fp_{n}F(S)_{\leq k}$-preenvelope. Similarly, it is proved that the class $(C\otimes_{R}S)fp_{n}F(S)_{\leq k}$ is  precovering.
\end{proof}
\begin{corollary}\label{3-14}
Let $S\geq R$ be an almost excellent extension. Then the following assertions are equivalent:
\begin{itemize}
\item[\emph{(i)}] Every $S$-module has a monic $(C\otimes_{R}S)fp_{n}I(S)_{\leq k}$-cover;
\item[\emph{(ii)}] Every $S^{op}$-module has an epic $(C\otimes_{R}S)fp_{n}F(S^{op})_{\leq k}$-envelope;
\item[\emph{(iii)}] Every quotient in $(C\otimes_{R}S)fp_{n}I(S)_{\leq k}$ is in $(C\otimes_{R}S)fp_{n}I(S)_{\leq k}$;
\item[\emph{(iv)}] Every submodule of   $(C\otimes_{R}S)fp_{n}F(S^{op})_{\leq k}$ is in $(C\otimes_{R}S)fp_{n}F(S^{op})_{\leq k}$.
\end{itemize}
Moreover, if $R$ is an $n$-coherent ring, then the above conditions are also equivalent to:
\begin{itemize}
\item[\emph{(v)}] The kernel of any $Cfp_{n}I(R)$-precover of any $R$-module is in $Cfp_{n}I(R)$;
\item[\emph{(vi)}] The cokernel of any $Cfp_{n}F(R^{op})$-preenvelope of any $R^{op}$-module is in $Cfp_{n}F(R^{op})$.
\end{itemize}
\end{corollary}

\begin{proof}
(i)$\Leftrightarrow$(iii).
First, we show that the class $(C\otimes_{R}S)fp_{n}I(S)_{\leq k}$ is closed under direct sums. Let $\{M_j\}_{j\in J}$ be a family of $S$-modules such that every $M_{j}\in(C\otimes_{R}S)fp_{n}I(S)_{\leq k}$. Then by Proposition \ref{1.p10p},  $M_{j}\in Cfp_{n}I(R)_{\leq k}$, and then by Proposition \ref{3-10}(i), $\bigoplus_{j\in J}M_{j}\in Cfp_{n}I(R)_{\leq k}$, and so by Proposition \ref{1.p10p}, $\bigoplus_{j\in J}M_{j}\in (C\otimes_{R}S)fp_{n}I(S)_{\leq k}$.
 So \cite[Proposition 4]{Z.JJG} shows that (i) and (iii) are equivalent.

(ii)$\Leftrightarrow$(iv).
The proof is similar to that of (i)$\Leftrightarrow$(iii) by using Propositions \ref{3-10}(ii), \ref{1.p10} and \cite[Theorem 2]{CD}.

(iii)$\Rightarrow$(iv).
Let $N\in (C\otimes_{R}S)fp_{n}F(S^{op})_{\leq k}$ and $N'$ be a submodule of $N$. From the short exact sequence
\[0\longrightarrow N'\longrightarrow N\longrightarrow N/N'\longrightarrow 0,\]
we get the short exact sequence
\[0\longrightarrow (N/N')^*\longrightarrow N^*\longrightarrow N'^*\longrightarrow 0.\]
By Propositions \ref{1.p10} and   \ref{3-8}(ii), $N\in Cfp_{n}F(R^{op})_{\leq k}$ if and only if $N^*\in Cfp_{n}I(R)_{\leq k}$ if and only if $N^{*}\in (C\otimes_{R}S)fp_{n}I(S)_{\leq k}$. Then by (iii) and Propositions \ref{1.p10p},  $N'^{*}\in (C\otimes_{R}S)fp_{n}I(S)_{\leq k}$ if and only if $N'^{*}\in Cfp_{n}I(R)_{\leq k}$, and consequently by Propositions \ref{3-8}(i) and \ref{1.p10}  , $N'\in Cfp_{n}F(R^{op})_{\leq k}$ if and only if $N'\in (C\otimes_{R}S)fp_{n}F(S^{op})_{\leq k}$. 

(iv)$\Rightarrow$(iii).
This is similar to that of (iii)$\Rightarrow$(iv).

(i)$\Rightarrow$(v).
Assume that $M$ is an $S$-module and that, by Theorem \ref{1.p124}, $f: F\longrightarrow M$ is a $(C\otimes_{R}S)fp_{n}I(S)_{\leq k}$-precover of $M$. Assume also that $g: E\longrightarrow M$ is a monic $(C\otimes_{R}S)fp_{n}I(S)_{\leq k}$-cover of $M$. Then \cite[Lemma 8.6.3]{EM} implies that $\operatorname{Ker}(f)\oplus E\cong F$.  By Proposition \ref{1.p10p}, $F\in Cfp_{n}I(R)_{\leq k}$, and so by Proposition \ref{3-10}(i),  $\operatorname{Ker}(f)\in  Cfp_{n}I(R)_{\leq k}$. Then $\operatorname{Ker}(f)\in (C\otimes_{R}S)fp_{n}I(S)_{\leq k}$ from Proposition \ref{1.p10p}.

(ii)$\Rightarrow$(vi).
The proof is similar to that of (i)$\Rightarrow$(v) by using the dual of \cite[Lemma 8.6.3]{EM}.

(vi)$\Rightarrow$(iv).
Assume that $N\in (C\otimes_{R}S)fp_{n}F(S^{op})_{\leq k}$ and that $N'$ is a submodule of $N$. Assume also that, by Theorem \ref{1.p125}, $f: N'\longrightarrow F$ is a $(C\otimes_{R}S)fp_{n}F(S^{op})_{\leq k}$-preenvelope of $N'$. Then we have the following commutative
diagram
\[\xymatrix{
&N'\ar@{=}[d]\ar[r]^{f}&F\ar[d]\ar[r]&\operatorname{Coker}(f)\ar[r]&0\\
0\ar[r]&N'\ar[r]&N&&
}\]
with exact rows. In particular, the sequence
\[0\longrightarrow N'\longrightarrow F\longrightarrow \operatorname{Coker}(f)\longrightarrow 0\]
is exact, and then by Remark \ref{3-2}(ii) and Corollary \ref{1.p1000}(ii), $N'\in (C\otimes_{R}S)fp_{n}F(S^{op})_{\leq k}$.

(v)$\Rightarrow$(iii).
The proof is similar to that of (vi)$\Rightarrow$(iv)  by using Corollary \ref{1.p1000}(i). 
\end{proof}
In next proposition, we investigate the homological behavior of Auslander and Bass classes under almost excellent extension of rings.
\begin{proposition}\label{1.p121}
Let $S\geq R$ be an almost excellent extension. Then the following assertions hold:
\begin{itemize}
\item[\emph{(i)}]
 If $M\in \mathcal{A}_{C}(R)$, then $(S\otimes_{R}M)\in \mathcal{A}_{C\otimes_{R}S}(S)$;
\item[\emph{(ii)}]
 If $M\in \mathcal{B}_{C}(R)$, then ${\rm Hom}_{R}(S,M)\in \mathcal{B}_{C\otimes_{R}S}(S)$.
\end{itemize}
\end{proposition} 
\begin{proof}
(i). There exists an exact sequence of $R$-modules
\[\cdots\longrightarrow F_{j+ 1}\longrightarrow F_j\longrightarrow F_{j- 1}\longrightarrow\cdots\longrightarrow F_1\longrightarrow F_0\longrightarrow C\longrightarrow 0,\]
where each $F_j$ is finitely generated and free for all $j\geq 0$. Since $M\in \mathcal{A}_{C}(R)$, we have the following exact sequence
\[\cdots\longrightarrow F_{j+ 1}\otimes_{R}M\longrightarrow F_j\otimes_{R}M\longrightarrow F_{j- 1}\otimes_{R}M\longrightarrow\cdots\longrightarrow F_1\otimes_{R}M\longrightarrow F_0\otimes_{R}M\longrightarrow C\otimes_{R}M\longrightarrow 0,\]
and since $S$ is flat $R$-module, we have the following commutative diagram
\[
\xymatrix{
S\otimes_{R}(F_{j+1}\otimes_{R}M)\ar[r]\ar[d]^{\cong}&S\otimes_{R}(F_{j}\otimes_{R}M)\ar[r]\ar[d]^{\cong}&\cdots \ar[r]&S\otimes_{R}(C\otimes_{R}M)\ar[r]\ar[d]^{\cong}&0\\
(F_{j+1}\otimes_{R}S)\otimes_{S}(S\otimes_{R}M)\ar[r]&(F_{j}\otimes_{R}S)\otimes_{S}(S\otimes_{R}M)\ar[r] &\cdots\ar[r]&(C\otimes_{R}S)\otimes_{S}(S\otimes_{R}M)\ar[r],&0,}
\]
and so  ${\rm Tor}^{S}_j(C\otimes_{R}S,S\otimes_{R}M)=0$  for any $j\geq 0$.

On the other hand, $C\otimes_{R}M\in\mathcal{B}_{C}(R)$ by \cite[Proposition 4.1]{HW}. So there exists the exact sequence

\[0\longrightarrow {\rm Hom}_{R}(C,C\otimes_{R}M)\longrightarrow\cdots\longrightarrow{\rm Hom}_{R}(F_{j},C\otimes_{R}M)\longrightarrow  {\rm Hom}_{R}(F_{j+1},C\otimes_{R}M)\longrightarrow\cdots,\]

and hence by \cite[Lemma 4.86]{Rot2}, we have the following commutative diagram:

\[
\xymatrix{
0\ar[r]&S\otimes_{R}{\rm Hom}_{R}(C,C\otimes_{R}M)\ar[r]\ar[d]^{\cong}&\cdots \ar[r]&S\otimes_{R}{\rm Hom}_{R}(F_{j+1},C\otimes_{R}M)\ar[d]^{\cong}&\\
0\ar[r]&{\rm Hom}_{R}(C,S\otimes_{R}(C\otimes_{R}M))\ar[r]\ar[d]^{\cong}&\cdots \ar[r]&{\rm Hom}_{R}(F_{j+1},S\otimes_{R}(C\otimes_{R}M))\ar[d]^{\cong}&\\
0\ar[r]&{\rm Hom}_{S}(C\otimes_{R}S,(C\otimes_{R}S)\otimes_{S}(S\otimes_{R}M))\ar[r]&\cdots \ar[r]&{\rm Hom}_{S}(F_{j+1}\otimes_{R}S,(C\otimes_{R}S)\otimes_{S}(S\otimes_{R}M)).&}
\]

Therefore, we deduce that  ${\rm Ext}_{S}^j(C\otimes_{R}S,(C\otimes_{R}S)\otimes_{S}(S\otimes_{R}M))=0,$   and also 
$$S\otimes_{R}M\cong S\otimes_{R}{\rm Hom}_{R}(C,C\otimes_{R}M)\cong{\rm Hom}_{S}(C\otimes_{R}S,(C\otimes_{R}S)\otimes_{S}(S\otimes_{R}M)).$$ Hence, it follows that $(S\otimes_{R}M)\in \mathcal{A}_{C\otimes_{R}S}(S).$

(ii). Let $M\in\mathcal{B}_{C}(R)$. Then by Proposition \ref{3-5-A-B}(ii), $M^{*}\in\mathcal{A}_{C}(R^{op})$. So $(S\otimes_{R^{op}}M^{*})\in \mathcal{A}_{C\otimes_{R^{op}}S}(S^{op})$  by (i). Since $S$ is a finitely presented $R$-module, \cite[Lemma 3.55]{Rot2} implies that ${\rm Hom}_{R^{op}}(S,M)^{*}\in \mathcal{A}_{C\otimes_{R^{op}}S}(S^{op})$. Consider the exact sequence $\mathcal{Y}=\cdots\longrightarrow F_1\longrightarrow F_0\longrightarrow C\longrightarrow 0$ of $R$-modules,
where each $F_j$ is finitely generated and free for all $j\geq 0$. Then by Lemma \ref{1.p1}, $\mathcal{Y}\otimes_{R}S$ is a $\mathcal{Y}\otimes_{R}S$-finitely presented, and then similar to the proof of Proposition \ref{3-5-A-B}(ii), ${\rm Hom}_{R}(S,M)\in \mathcal{B}_{C\otimes_{R}S}(S)$.
\end{proof}

\begin{corollary}\label{1.p1211}
Let $S\geq R$ be an almost excellent extension. Then the following assertions hold:
\begin{itemize}
\item[\emph{(i)}]
 $fp_{n}F(S)_{\leq k}\subseteq \mathcal{A}_{C\otimes_{R}S}(S)$;
\item[\emph{(ii)}]
 $fp_{n}I(S)_{\leq k}\subseteq \mathcal{B}_{C\otimes_{R}S}(S)$.
\end{itemize}
\end{corollary} 
\begin{proof}
(i). Let $M\in fp_{n}F(S)_{\leq k}$. Then there exists an exact sequence
\[0\longrightarrow Y_k\longrightarrow Y_{k- 1} \longrightarrow \cdots\longrightarrow X_{ 1}\longrightarrow X_0\longrightarrow M\longrightarrow0\]
of $S$-modules with each $X_i\in fp_nF(S)$ for all $0\leq i\leq k$. By \cite[Proposition 3.2]{JZW}, $X_i\in fp_nF(R)$. So we obtain that $M\in fp_{n}F(R)_{\leq k}$. Thus by Lemma \ref{3-5}(ii), $M\in \mathcal{A}_{C}(R)$, and so by Proposition \ref{1.p121}(i),  $(S\otimes_{R}M)\in \mathcal{A}_{C\otimes_{R}S}(S)$. By \cite[Lemma 1.1]{WX}, we see that   $S$-module $M$
is isomorphic to a direct summand of $S\otimes_{R}M$. Then \cite[Proposition 4.2]{HW} implies that $M\in\mathcal{A}_{C\otimes_{R}S}(S)$.

(ii).This is similar to the proof of (i).
\end{proof}
\begin{lemma}\label{4-2q}
 Let $S\geq R$ be an almost excellent extension. Then the following assertions hold true:
\begin{itemize}
\item[\emph{(i)}] $(C\otimes_{R}S)fp_{n}I(S)_{\leq k}\subseteq \mathcal{A}_{C\otimes_{R}S}(S)$;
\item[\emph{(ii)}] $(C\otimes_{R}S)fp_{n}F(S)_{\leq k}\subseteq \mathcal{B}_{C\otimes_{R}S}(S)$.
\end{itemize}
\end{lemma}

\begin{proof}
(i).  Assume that $M\in(C\otimes_{R}S)fp_{n}I(S)_{\leq k}$. Then by Proposition \ref{1.p10p}, $M\in Cfp_{n}I(R)_{\leq k}$, and so $M\in \mathcal{A}_{C}(R)$ by Lemma \ref{4-2}(i). Thus by Proposition \ref{1.p121}(i), $(S\otimes_{R}M)\in \mathcal{A}_{C\otimes_{R}S}(S)$. By \cite[Lemma 1.1]{WX}, $M$ is isomorphic to a direct summand of 
$S\otimes_{R}M$, and consequently by \cite[Proposition 4.2]{HW}, $M\in \mathcal{A}_{C\otimes_{R}S}(S)$.

(ii). This is similar to the first part.
\end{proof}
In the following, we investigate Foxby equivalence relative to the class $(C\otimes_{R}S)fp_{n}I(S)_{\leq k}$ with the class $fp_nI(S)_{\leq k}$ and  the class $(C\otimes_{R}S)fp_{n}F(S)_{\leq k}$ with the class $fp_nF(S)_{\leq k}$, where $S\geq R$ is an almost excellent extension.
\begin{proposition}\label{4-3i}
Let $S\geq R$ be an almost excellent extension. Then we have the following equivalences of categories:
\begin{itemize}
\item[\emph{(i)}]
$\xymatrix@C=3cm{(C\otimes_{R}S)fp_nI(S)_{\leq k}\ar@<1ex>[r]^{(C\otimes_{R}S)\otimes_{S} -}_{\sim}&\ar@<1ex>[l]^{\operatorname {Hom}_{S}(C\otimes_{R}S, -)}fp_{n}I(S)_{\leq k};}$
\item[\emph{(ii)}]
$\xymatrix@C=3cm{fp_nF(S)_{\leq k}\ar@<1ex>[r]^{(C\otimes_{R}S)\otimes_{S} -}_{\sim}&\ar@<1ex>[l]^{\operatorname {Hom}_{S}(C\otimes_{R}S, -)}(C\otimes_{R}S)fp_{n}F(S)_{\leq k}.}$
\end{itemize}
\end{proposition}

\begin{proof}
(i). Let $M\in (C\otimes_{R}S)fp_nI(S)_{\leq k}$. Then there exists an exact sequence
\[0\longrightarrow M\longrightarrow I_0 \longrightarrow I_1\longrightarrow \cdots\longrightarrow I_{k- 1}\longrightarrow I_k\longrightarrow 0\]
of $S$-modules with each $I_i\in (C\otimes_{R}S)fp_nI(S)$ for all $0\leq i\leq k$. By Proposition \ref{1.p10p}, each $I_i\in Cfp_nI(R)$, and so by Proposition \ref{4-3}(i) and \cite[Proposition 3.2]{JZW}, $C\otimes_{R}I_{i}\in fp_nI(R)$ if and only if $C\otimes_{R}I_{i}\in fp_nI(S)$. On the other hand,   
by Proposition \ref{1.p10p}, $M\in Cfp_nI(R)_{\leq k}$, and then by Lemma \ref{4-2}(i), $M$ and  $I_{i}$ are in $\mathcal{A}_{C}(R)$. So, there exists  exact sequence
\[0\longrightarrow C\otimes_{R}M\longrightarrow C\otimes_{R}I_0 \longrightarrow C\otimes_{R}I_1\longrightarrow \cdots\longrightarrow C\otimes_{R}I_{k- 1}\longrightarrow C\otimes_{R}I_k\longrightarrow 0\]
of $S$-modules with each $C\otimes_{R}I_i\in Cfp_nI(S)$ for all $0\leq i\leq k$, and hence $(C\otimes_{R}S)\otimes_{S}M\cong C\otimes_{R}M\in fp_nI(S)_{\leq k}$.

Also,  $M\in\mathcal{A}_{C\otimes_{R}S}(S)$ by Lemma \ref{4-2q}(i). So we have $M\cong \operatorname {Hom}_{S}(C\otimes_{R}S,(C\otimes_{R}S)\otimes_{S}M)$. 

Now, let $N\in fp_{n}I(S)_{\leq k}$. Then there exists an exact sequence
\[0\longrightarrow N\longrightarrow X_0 \longrightarrow X_1\longrightarrow \cdots\longrightarrow X_{k- 1}\longrightarrow X_k\longrightarrow 0\]
of $S$-modules with each $X_i\in fp_nI(S)$ for all $0\leq i\leq k$. By \cite[Proposition 3.2]{JZW}, $X_i\in fp_nI(R)$. So we get that $N\in fp_{n}I(R)_{\leq k}$. Thus by Proposition \ref{4-3}(i), ${\rm Hom}_{R}(C,N)\in Cfp_{n}I(R)_{\leq k}$. We have 
${\rm Hom}_{S}(C\otimes_{R}S,N)\cong {\rm Hom}_{R}(C,{\rm Hom}_{S}(S,N))\cong{\rm Hom}_{R}(C,N)$. Hence ${\rm Hom}_{S}(C\otimes_{R}S,N)\in Cfp_{n}I(R)_{\leq k}$, and then by  Proposition \ref{1.p10p}, ${\rm Hom}_{S}(C\otimes_{R}S,N)\in (C\otimes_{R}S)fp_nI(S)_{\leq k}$.

(ii). This is similar to that of (i).
\end{proof}
In the following, we give equivalent conditions with modules of the classes $\mathcal{A}_{C}(R)$ and  $\mathcal{B}_{C}(R)$ under almost excellent extension of rings.

\begin{proposition}\label{1.p10p-A}
Let $S\geq R$ be an almost excellent extension and $M$ an $S$-module. Then the following assertions are equivalent:
\begin{itemize}
\item[\emph{(i)}]
 $M\in\mathcal{A}_{C}(R)$;
\item[\emph{(ii)}]
 $(S\otimes_{R}M)\in\mathcal{A}_{C\otimes_{R}S}(S)$;
\item[\emph{(iii)}]
 $M\in\mathcal{A}_{C\otimes_{R}S}(S)$.
\end{itemize}
\end{proposition} 
\begin{proof}
(i)$\Longrightarrow $(ii). It is clear by Proposition \ref{1.p121}(1).

(ii)$\Longrightarrow $(iii). By \cite[Lemma 1.1]{WX}, $_SM$ is isomorphic to a direct summand of $S$-module $S\otimes_{R}M$. Then by \cite[Proposition 4.2(1)]{HW},  $M\in\mathcal{A}_{C\otimes_{R}S}(S)$.

(iii)$\Longrightarrow $(i). Let $M\in\mathcal{A}_{C\otimes_{R}S}(S)$. Then ${\rm Tor}_{j}^S(C\otimes_{R}S,M)=0$ for any $j\geq 1$. So,
we have the following commutative diagram:

\[
\xymatrix{
\cdots \ar[r]&(F_{1}\otimes_{R}S)\otimes_{S}M\ar[r]\ar[d]^{\cong}&(F_{0}\otimes_{R}S)\otimes_{S}M\ar[r]\ar[d]^{\cong}&(C\otimes_{R}S)\otimes_{S}M\ar[r]\ar[d]^{\cong}&0\\
\cdots \ar[r]&F_{1}\otimes_{R}M\ar[r]&F_{0}\otimes_{R}M\ar[r]& C\otimes_{R}M\ar[r]&0,\\
}
\]
where the first line is exact by (iii), and so the second line is also exact, and then  ${\rm Tor}_{j}^R(C,M)=0$ for any $j\geq 1$.

On the other hand, ${\rm Ext}_{S}^j(C\otimes_{R}S,(C\otimes_{R}S)\otimes_{S}M)=0$ for any $j\geq 1$. Then, we have the following commutative diagram:
\[
\xymatrix{
0\ar[r]&{\rm Hom}_{S}(C\otimes_{R}S,(C\otimes_{R}S)\otimes_{S}M)\ar[r]\ar[d]^{\cong}&{\rm Hom}_{S}(C\otimes_{R}S,(C\otimes_{R}S)\otimes_{S}M)\ar[r]\ar[d]^{\cong} &\cdots\\
0\ar[r]&{\rm Hom}_{S}(C\otimes_{R}S,C\otimes_{R}M)\ar[r]\ar[d]^{\cong}&{\rm Hom}_{S}(C\otimes_{R}S,C\otimes_{R}M)\ar[r]\ar[d]^{\cong}&\cdots\\
0\ar[r]&{\rm Hom}_{R}(C,C\otimes_{R}M)\ar[r]&{\rm Hom}_{R}(C,C\otimes_{R}M)\ar[r]&\cdots,}
\]
where the first and second lines are exact by (iii), and so the third  line is also exact, and then  ${\rm Ext}^{j}_R(C,M)=0$ for any $j\geq 1$.

Also by (iii) and \cite[Theorem 2.75]{Rot2}, we have $$M\cong {\rm Hom}_{S}(C\otimes_{R}S,(C\otimes_{R}S)\otimes_{S}M)\cong{\rm Hom}_{S}(C\otimes_{R}S,C\otimes_{R}M)\cong{\rm Hom}_{R}(C,C\otimes_{R}M).$$ Consequently, $M\in\mathcal{A}_{C}(R)$. 
\end{proof}

\begin{proposition}\label{1.p10p-B}
Let $S\geq R$ be an almost excellent extension and $M$ an $S$-module. Then the following assertions are equivalent:
\begin{itemize}
\item[\emph{(i)}]
 $M\in\mathcal{B}_{C}(R)$;
\item[\emph{(ii)}]
 ${\rm Hom}_{R}(S,M)\in\mathcal{B}_{C\otimes_{R}S}(S)$;
\item[\emph{(iii)}]
 $M\in\mathcal{B}_{C\otimes_{R}S}(S)$.
\end{itemize}
\end{proposition} 
\begin{proof}
 This is similar to the proof of Proposition \ref{1.p10p-A}.
\end{proof}
Under chang of rings, Auslander and Bass classes  are equivalent under the
pair of functors.
\begin{proposition}\label{1.p10p-AB}
Let $S\geq R$ be an almost excellent extension. Then there are
equivalences of categories:
\[\xymatrix@C=3cm{\mathcal{A}_{C\otimes_{R}S}(S)\ar@<1ex>[r]^{(C\otimes_{R}S)\otimes_{S} -}_{\sim}&\ar@<1ex>[l]^{\operatorname{Hom}_{S}(C\otimes_{R}S, -)}\mathcal{B}_{C\otimes_{R}S}(S)}\]
\end{proposition}
\begin{proof}
 By Proposition \ref{1.p10p-A}, $M\in\mathcal{A}_{C\otimes_{R}S}(S)$ if and only if $M\in \mathcal{A}_{C}(R)$. Then by \cite[Proposition 4.1]{HW}, $(C\otimes_{R}M)\in \mathcal{B}_{C}(R)$, and so  $(C\otimes_{R}S)\otimes_{S}M\cong (C\otimes_{R}M)\in\mathcal{B}_{C\otimes_{R}S}(S)$ by Proposition \ref{1.p10p-B}.  Also, we have $M\cong {\rm Hom}_{R}(C,C\otimes_{R}M)\cong{\rm Hom}_{S}(C\otimes_{R}S,(C\otimes_{R}S)\otimes_{S}M)$.
 
 On the other hand, By Proposition \ref{1.p10p-B}, $N\in \mathcal{B}_{C\otimes_{R}S}(S)$ if and only if $N\in \mathcal{B}_{C}(R)$. Thus by \cite[Proposition 4.1]{HW}, ${\rm Hom}_{R}(C,N)\in \mathcal{A}_{C}(R)$, and so ${\rm Hom}_{S}(C\otimes_{R}S, N)\cong {\rm Hom}_{R}(C,N)\in\mathcal{A}_{C\otimes_{R}S}(S)$ by Proposition \ref{1.p10p-A} and \cite[Theorem 2.75]{Rot2}. Also, we have $$N\cong C\otimes_{R}{\rm Hom}_{R}(C,N)\cong (C\otimes_{R}S)\otimes_{S}{\rm Hom}_{S}(C\otimes_{R}S,N).$$
\end{proof}

By using Corollary \ref{1.p1211}, Lemma \ref{4-2q} and Propositions \ref{1.p121},  \ref{4-3i}, \ref{1.p10p-AB}, we get Foxby Equivalence  under an almost excellent extension: 

\begin{theorem}\emph{(Foxby Equivalence under  almost excellent extension of rings)}\label{4-5-S-R}
 Then we have the following equivalences of categories:
\[\xymatrix@C=3cm{
fp_nF(S)\ar@<1ex>[r]^{(C\otimes_{R}S)\otimes_{S} -}_{\sim}\ar@{^{(}->}[d]&\ar@<1ex>[l]^{\operatorname{Hom}_{S}(C\otimes_{R}S, -)}(C\otimes_{R}S)fp_nF(S)\ar@{^{(}->}[d]\\
fp_nF(S)_{\leq k}\ar@<1ex>[r]^{(C\otimes_{R}S)\otimes_{S} -}_{\sim}\ar@{^{(}->}[d]&\ar@<1ex>[l]^{\operatorname{Hom}_{S}(C\otimes_{R}S, -)}(C\otimes_{R}S)fp_nF(S)_{\leq k}\ar@{^{(}->}[d]\\
\mathcal{A}_{C\otimes_{R}S}(S)\ar@<1ex>[r]^{(C\otimes_{R}S)\otimes_{S} -}_{\sim}& \ar@<1ex>[l]^{\operatorname{Hom}_{S}(C\otimes_{R}S, -)}\mathcal{B}_{C\otimes_{R}S}(S)\\
(C\otimes_{R}S)fp_{ n}I(S)_{\leq k}\ar@<1ex>[r]^{(C\otimes_{R}S)\otimes_{S} -}_{\sim}\ar@{^{(}->}[u]&\ar@ <1ex>[l]^{\operatorname{Hom}_{S}(C\otimes_{R}S, -)}fp_nI(S)_{\leq k}\ar@{^{(}->}[u]\\
(C\otimes_{R}S)fp_{ n}I(S)\ar@<1ex>[r]^{(C\otimes_{R}S)\otimes_{S} -}_{\sim}\ar@{^{(}->}[u]&\ar@<1ex>[l]^{\operatorname{Hom}_{S}(C\otimes_{R}S, -)}fp_nI(S)\ar@{^{(}->}[u].
}\]
\end{theorem}

\bibliographystyle{amsplain}

\end{document}